\documentclass[10pt,a4paper]{amsart}
\usepackage{amsmath,amsfonts,amsthm,amssymb,bm,graphicx,mathtools}
\usepackage[stretch=10]{microtype}
\usepackage[unicode]{hyperref}
\usepackage{xcolor}
\definecolor{dark-red}{rgb}{0.4,0.15,0.15}
\definecolor{dark-blue}{rgb}{0.15,0.15,0.4}
\definecolor{medium-blue}{rgb}{0,0,0.5}
\hypersetup{
	pdftitle={Spectral Multiplicity for Maa\ss{} Newforms of Non-Squarefree Level},
	pdfauthor={Peter Humphries},
	pdfnewwindow=true,   
	colorlinks, linkcolor={dark-red},
	citecolor={dark-blue}, urlcolor={medium-blue}
}
\makeatletter
\newcommand*{\defeq}{\mathrel{\rlap{%
                     \raisebox{0.3ex}{$\m@th\cdot$}}%
                     \raisebox{-0.3ex}{$\m@th\cdot$}}%
                     =}
\makeatother
\makeatletter
\newcommand*{\eqdef}{=\mathrel{\rlap{%
                     \raisebox{0.3ex}{$\m@th\cdot$}}%
                     \raisebox{-0.3ex}{$\m@th\cdot$}}%
                     }
\makeatother
\newcommand\A{\mathbb{A}}
\newcommand\Ac{\mathcal{A}}
\newcommand\BB{\mathcal{B}}
\newcommand\Cb{\mathbb{C}}
\newcommand\DD{\mathcal{D}}
\newcommand\dd{\mathfrak{d}}
\newcommand\dee{\partial}
\newcommand\ds{\mathrm{ds}}

\newcommand\Ff{\mathfrak{F}}
\newcommand\ff{\mathfrak{f}}
\newcommand\gl{\mathfrak{gl}}
\newcommand\Hb{\mathbb{H}}
\newcommand\kbf{\boldsymbol{k}}
\newcommand\N{\mathbb{N}}

\newcommand\OO{\mathcal{O}}
\newcommand\pp{\mathfrak{p}}
\newcommand\Pp{\mathfrak{P}}
\newcommand\ps{\mathrm{ps}}
\newcommand\Q{\mathbb{Q}}
\newcommand\qq{\mathfrak{q}}
\newcommand\Quad{\mathrm{quad}}
\newcommand\R{\mathbb{R}}
\newcommand\rr{\mathfrak{r}}
\newcommand\sbf{\boldsymbol{s}}
\newcommand\Sc{\mathcal{S}}
\newcommand\St{\mathrm{St}}
\newcommand\Tbf{\boldsymbol{T}}

\newcommand\Xx{\mathfrak{X}}
\newcommand\Z{\mathbb{Z}}
\DeclareMathOperator{\cond}{cond}
\DeclareMathOperator{\Gal}{Gal}
\DeclareMathOperator{\GL}{GL}
\DeclareMathOperator{\mon}{mon}
\DeclareMathOperator{\nonmon}{nonmon}
\DeclareMathOperator{\PGL}{PGL}
\DeclareMathOperator{\PSL}{PSL}
\DeclareMathOperator{\sgn}{sgn}
\DeclareMathOperator{\SL}{SL}

\DeclareMathOperator{\vol}{vol}
\numberwithin{equation}{section}
\newtheorem{theorem}[equation]{Theorem}
\newtheorem{conjecture}[equation]{Conjecture}

\newtheorem{lemma}[equation]{Lemma}
\newtheorem{proposition}[equation]{Proposition}
\newtheorem*{question}{Question}
\theoremstyle{remark}

\begin{document}

\title{Spectral Multiplicity for Maa\ss{} Newforms of Non-Squarefree Level}

\author{Peter Humphries}

%\date{\today}

\address{Department of Mathematics, University College London, Gower Street, London WC1E 6BT}

\email{\href{mailto:pclhumphries@gmail.com}{pclhumphries@gmail.com}}

\keywords{spectrum, Laplacian, multiplicity, Maa\ss{} form}

\subjclass[2010]{11F12 (primary); 11F72 (secondary).}

\begin{abstract}
We show that if a positive integer $q$ has $s(q)$ odd prime divisors $p$ for which $p^2$ divides $q$, then a positive proportion of the Laplacian eigenvalues of Maa\ss{} newforms of weight $0$, level $q$, and principal character occur with multiplicity at least $2^{s(q)}$. Consequently, the new part of the cuspidal spectrum of the Laplacian on $\Gamma_0(q) \backslash \Hb$ cannot be simple for any odd non-squarefree integer $q$. This generalises work of Str\"{o}mberg, who proved this for $q = 9$ by different methods.
\end{abstract}

\maketitle

\section{Introduction}

\subsection{Main Results}

Given a positive integer $q$, we denote by $\Gamma_0(q)$ the congruence subgroup of $\SL_2(\Z)$ consisting of matrices whose lower left entry is divisible by $q$. The cuspidal spectrum of the Laplacian
\[\Delta = -y^2 \left(\frac{\dee^2}{\dee x^2} + \frac{\dee^2}{\dee y^2}\right)\]
on $\Gamma_0(q) \backslash \Hb$ is the set of eigenvalues of the Laplacian acting on Maa\ss{} cusp forms on $\Gamma_0(q) \backslash \Hb$ of weight $0$ and character $\chi_{0(q)}$, the principal Dirichlet character modulo $q$. It is known that the multiplicity of an eigenvalue $\lambda$ of the Laplacian on $\Gamma_0(q) \backslash \Hb$ is at most $C_q \sqrt{\lambda} / \log \lambda$ for some constant $C_q$ dependent only on $q$; cf.\ \cite[\S 4]{Sarnak}.

For $q = 1$, Cartier \cite{Cartier} conjectured that the cuspidal spectrum is simple, based on limited numerical evidence; more numerical calculations of Steil \cite{Steil} gave further support to this conjecture. For $q > 1$, the cuspidal spectrum is not simple due to newforms and oldforms having the same eigenvalue. Nevertheless, one may well ask whether this is the only cause of spectral multiplicity, so that the so-called new part of the cuspidal spectrum consisting of eigenvalues of the Laplacian associated to newforms is simple.

\begin{question}
Is the new part of the cuspidal spectrum of the Laplacian on $\Gamma_0(q) \backslash \Hb$ simple?
\end{question}

This folklore conjecture is certainly expected to be true for squarefree $q$; see, for example, the work of Bolte and Johansson \cite{Bolte1,Bolte2} and Str\"{o}mbergsson \cite{Strombergsson} on the spectral correspondence between the new part of the cuspidal spectrum of the Laplacian on $\Gamma_0(q) \backslash \Hb$ and eigenvalues of the automorphic Laplacian for the group of units of norm one in a maximal order in an indefinite quaternion division algebra over $\Q$, and in particular the discussion \cite[p.\ 1154]{Bolte2}. Furthermore, the $L$-functions and modular forms database \cite{LMFDB} contains a list of numerical calculations of small eigenvalues of the Laplacian of Maa\ss{} newforms on $\Gamma_0(q) \backslash \Hb$ for low values of $q$, and no eigenvalues of multiplicity greater than one are listed whenever $q$ is squarefree.

On the other hand, it is known by purely representation-theoretic means that the cuspidal spectrum of the Laplacian on $\Gamma(p) \backslash \Hb$, where $\Gamma(p)$ is the principal congruence subgroup modulo an odd prime $p$, contains infinitely many eigenvalues with multiplicity at least $\left(p + (-1)^{(p - 1)/2}\right) / 2$ \cite{Randol}; we give another explanation of this case of spectral multiplicity in \hyperref[TwistsGamma1(q)Gamma(q)sect]{Section \ref*{TwistsGamma1(q)Gamma(q)sect}}, and in particular show that the part of the spectrum of the Laplacian on $\Gamma(p) \backslash \Hb$ corresponding to newforms contains eigenvalues occurring with multiplicity at least $(p - 3)/2$.

In this article, we show that the new part of the cuspidal spectrum of the Laplacian on $\Gamma_0(q) \backslash \Hb$ is never simple when $q$ has an odd prime divisor $p$ for which $p^2$ divides $q$. More precisely, we prove the following.

\begin{theorem}\label{introthm}
Let $q$ be a positive integer, and let $s(q)$ denote the number of distinct odd primes $p$ for which $p^2$ divides $q$. Then a set of eigenvalues of the new part of the cuspidal spectrum of the Laplacian on $\Gamma_0(q) \backslash \Hb$ with density
\[\prod_{p^2 \parallel q} \left(1 - \frac{p}{p^2 - p - 1}\right)\]
has multiplicity at least $2^{s(q)}$.
\end{theorem}

In particular, the new part of the cuspidal spectrum of the Laplacian on $\Gamma_0(q) \backslash \Hb$ will contain eigenvalues with high multiplicity if $q$ is highly composite and non-squarefree.

\subsection{Method of Proof}

Denote by $\Ac_0(q,\chi)$ the vector space of Maa\ss{} cusp forms of weight $0$, level $q$, and character $\chi$, and by $\Ac_0(q,\chi;\lambda)$ the subspace of $\Ac_0(q,\chi)$ consisting of Maa\ss{} cusp forms with Laplacian eigenvalue $\lambda = 1/4 + t^2 > 0$, where the spectral parameter $it$ is in $i\R \cup (-1/2,1/2)$. In each case, we let $\Ac_0^{\ast}(q,\chi)$ and $\Ac_0^{\ast}(q,\chi;\lambda)$ denote the subspaces spanned by newforms of level $q$, and by $\BB_0^{\ast}(q,\chi)$ and $\BB_0^{\ast}(q,\chi;\lambda)$ the bases of these spaces consisting of Hecke-normalised newforms, so that the Fourier coefficients $\rho_f(n)$ of a newform with $n \geq 1$ are precisely its Hecke eigenvalues $\lambda_f(n)$ of the $n$-th Hecke operator
\[(T_n f)(z) \defeq \frac{1}{\sqrt{n}} \chi(a) \sum_{ad = n} \sum_{n \hspace{-.25cm} \pmod{d}} g\left(\frac{az + b}{d}\right).\]
When $\chi = \chi_{0(q)}$ is the principal character modulo $q$, we simply write $\Ac_0^{\ast}\left(\Gamma_0(q)\right)$, $\Ac_0^{\ast}\left(\Gamma_0(q);\lambda\right)$, $\BB_0^{\ast}\left(\Gamma_0(q)\right)$, and $\BB_0^{\ast}\left(\Gamma_0(q);\lambda\right)$.

Given a primitive Dirichlet character $\psi$ modulo $r$ and a Maa\ss{} newform $f \in \BB_0^{\ast}(q,\chi)$ with Hecke eigenvalues $\lambda_f(p)$ at each prime $p$, the twist of $f$ by $\psi$ is the \emph{newform} $f \otimes \psi$ with Hecke eigenvalues equal to $\psi(p) \lambda_f(p)$ for every prime $p$ not dividing $rq$; this Maa\ss{} newform is of weight $0$, level dividing $r^2 q$, and character induced by the same primitive character as $\psi^2 \chi$. Crucially, twisting by a character leaves the Laplacian eigenvalue of a Maa\ss{} form unchanged.

The idea of proof of \hyperref[introthm]{Theorem \ref*{introthm}} is to show the existence of Maa\ss{} newforms that are level- and character-invariant under twisting by certain quadratic Dirichlet characters, and then to show that these twists are often distinct from the original newform; this then forces spectral multiplicity. Indeed, the numerical calculations in \cite{LMFDB} show that for each $q \in \{9,16,25,32,36,49,81,100\}$, there exist pairs of Maa\ss{} newforms $f,g \in \BB_0^{\ast}\left(\Gamma_0(q)\right)$ with identical Laplacian eigenvalue that are related by twisting; note that the numerical calculations do not include $q$ equal to $64$ or any nonprime greater than $100$, while for $q$ equal to $4$ or $8$, it is unclear whether eigenvalues occur with multiplicity greater than one.

The first result we prove is the following characterisation of Maa\ss{} newforms of principal character that are level-invariant under twisting by a quadratic character.

\begin{theorem}\label{twistthmclassical}
Let $q$ be a positive integer, and let $p$ be an odd prime dividing $q$. Let $\chi_p$ denote the unique quadratic Dirichlet character modulo $p$. Then if $p \parallel q$,
\[\BB_0^{\ast}\left(\Gamma_0(q);\lambda\right) \cap \left(\BB_0^{\ast}\left(\Gamma_0(q);\lambda\right) \otimes \chi_p\right) = \emptyset.\]
If $p^2 \parallel q$,
\begin{multline*}
\BB_0^{\ast}\left(\Gamma_0(q);\lambda\right) \cap \left(\BB_0^{\ast}\left(\Gamma_0(q);\lambda\right) \otimes \chi_p\right)	\\
= \BB_0^{\ast}\left(\Gamma_0(q);\lambda\right) \setminus \left( \left(\BB_0^{\ast}\left(\Gamma_0\left(\frac{q}{p}\right);\lambda\right) \otimes \chi_p\right) \cup \left(\BB_0^{\ast}\left(\Gamma_0\left(\frac{q}{p^2}\right);\lambda\right) \otimes \chi_p\right)\right).
\end{multline*}
If $p^3 \mid q$,
\[\BB_0^{\ast}\left(\Gamma_0(q);\lambda\right) \cap \left(\BB_0^{\ast}\left(\Gamma_0(q);\lambda\right) \otimes \chi_p\right) = \BB_0^{\ast}\left(\Gamma_0(q);\lambda\right).\]
\end{theorem}

Here by $\BB_0^{\ast}\left(\Gamma_0(q);\lambda\right) \otimes \chi_p$, we mean the set $\{f \otimes \chi_p : f \in \BB_0^{\ast}\left(\Gamma_0(q);\lambda\right)\}$.

\hyperref[twistthmclassical]{Theorem \ref*{twistthmclassical}} is proved in \hyperref[prooftwistthmsect]{Section \ref*{prooftwistthmsect}}, and then generalised to allow twists by the unique primitive quadratic Dirichlet character $\chi_r$ modulo $r$ for any odd squarefree divisor $r$ of $q$. We first transport this problem to the ad\`{e}lic setting, for then one can resolve this locally by studying generic irreducible admissible representations of $\GL_2(\Q_p)$. Moreover, we can generalise \hyperref[twistthmclassical]{Theorem \ref*{twistthmclassical}} to cuspidal automorphic forms on $\GL_2$ over any number field at one fell swoop; see \hyperref[twistthmadelic]{Theorem \ref*{twistthmadelic}}. By the well-known correspondence between newforms and automorphic representations (see, for example, \cite[Theorem 13.8.9]{Goldfeld}), it suffices to characterise the cuspidal automorphic representations $\pi$ of $\GL_2(\A_{\Q})$ whose archimedean component is a weight $0$ principal series representation such that $\pi$ is conductor-invariant under twisting by a quadratic Hecke character of $\Q^{\times} \backslash \A_{\Q}^{\times}$. We cover in \hyperref[backgroundsect]{Section \ref*{backgroundsect}} the relationship between automorphic representations and representations of local fields and follow this with a classification of the behaviour of such local representations under twisting by quadratic characters.

Via the Weyl law for $\Gamma_0(q) \backslash \Hb$, \hyperref[twistthmclassical]{Theorem \ref*{twistthmclassical}} shows that a positive proportion of Maa\ss{} newforms of non-squarefree level $q$ and principal character are level-invariant under twisting by a quadratic Dirichlet character. This, however, does not force spectral multiplicity for $\Gamma_0(q) \backslash \Hb$: we must eliminate the possibility that $f \otimes \chi_p = f$ for every newform $f \in \BB_0^{\ast}\left(\Gamma_0(q);\lambda\right) \cap \left(\BB_0^{\ast}\left(\Gamma_0(q);\lambda\right) \otimes \chi_p\right)$. Maa\ss{} newforms that satisfy $f \otimes \chi = f$ for some Dirichlet character $\chi$ are said to be monomial or dihedral or of CM-type, though the latter nomenclature is somewhat incongruous due to the fact that these newforms do not arise from CM-fields, while the former nomenclature stems from the fact that the cuspidal automorphic representation associated to $f$ is the automorphic induction of a character. In \hyperref[monomialsect]{Section \ref*{monomialsect}}, we give an upper bound for the number of monomial newforms on $\Gamma_0(q) \backslash \Hb$ with Laplacian eigenvalue at most $T^2$ with $T$ tending to infinity.

\hyperref[Weyllawsect]{Section \ref*{Weyllawsect}} deals with the Weyl law for Maa\ss{} newforms. We show that the monomial newforms have density zero in the set of all newforms of level $q$ and principal character. This allows us to count the number of Maa\ss{} newforms of level $q$ and principal character whose twist by a quadratic Dirichlet character $\chi_r$ is a different Maa\ss{} newform of the same level and character; we denote by $\BB_0^{\ast}\left(\Gamma_0(q)\right)_{\nonmon(\chi_r)}$ the subset of $\BB_0^{\ast}\left(\Gamma_0(q)\right)$ of such newforms.

\begin{theorem}\label{Qratiothm}
Let $q$ and $r$ be positive integers with $r > 1$ odd and squarefree. Let $\chi_r$ denote the unique primitive quadratic Dirichlet character modulo $r$. Then if $r^2$ does not divide $q$,
\[\BB_0^{\ast}\left(\Gamma_0(q)\right)_{\nonmon(\chi_r)} = \emptyset,\]
whereas if $r^2$ divides $q$, we have that
\[\frac{\# \left\{f \in \BB_0^{\ast}\left(\Gamma_0(q)\right)_{\nonmon(\chi_r)} : \lambda_f \leq T^2\right\}}{\#\left\{f \in \BB_0^{\ast}\left(\Gamma_0(q)\right) : \lambda_f \leq T^2\right\}} = \prod_{\substack{p \mid r \\ p^2 \parallel q}} \left(1 - \frac{p}{p^2 - p - 1}\right) + o_q(1)\]
as $T$ tends to infinity, where the error term depends only on $q$.

Moreover, the same holds if we replace $\BB_0^{\ast}\left(\Gamma_0(q)\right)_{\nonmon(\chi_r)}$ by
\[\bigcap_{\substack{r' \mid r \\ r' > 1}} \BB_0^{\ast}\left(\Gamma_0(q)\right)_{\nonmon(\chi_{r'})}.\]
\end{theorem}

Again, we prove this result in more generality; see \hyperref[Fratiothm]{Theorem \ref*{Fratiothm}}. The key tool in this proof is the explicit Weyl law for $\GL_2$ over any number field due to Palm \cite{Palm}.

As $f$ and $f \otimes \chi_r$ have the same Laplacian eigenvalue $\lambda_f$, this means that every Laplacian eigenvalue corresponding to a newform in $\BB_0^{\ast}\left(\Gamma_0(q)\right)_{\nonmon(\chi_r)}$ has multiplicity at least $2$. In particular, when $q = p^3$ is the cube of an odd prime, then \hyperref[twistthmclassical]{Theorems \ref*{twistthmclassical}} and \ref{Qratiothm} imply that eigenvalues with multiplicity at least $2$ constitute almost all of the new part of the cuspidal spectrum of the Laplacian on $\Gamma_0(q) \backslash \Hb$, with the only eigenvalues having multiplicity $1$ corresponding to monomial newforms.

More generally, the set
\[\bigcap_{\substack{r' \mid r \\ r' > 1}} \BB_0^{\ast}\left(\Gamma_0(q)\right)_{\nonmon(\chi_{r'})}\]
remains unchanged under twisting by each unique primitive quadratic Dirichlet character $\chi_{r'}$ modulo $r'$ with $r' > 1$ a divisor of $r$. As this set does not contain any monomial newforms, we may complete the proof of \hyperref[introthm]{Theorem \ref*{introthm}}.

\begin{proof}[Proof of \texorpdfstring{\hyperref[introthm]{Theorem \ref*{introthm}}}{Theorem \ref{introthm}}]
Set $r$ to be the largest odd squarefree integer such that $r^2$ divides $q$. Then every newform $f$ in
\[\bigcap_{\substack{r' \mid r \\ r' > 1}} \BB_0^{\ast}\left(\Gamma_0(q)\right)_{\nonmon(\chi_{r'})}\]
has the same Laplacian eigenvalue as $f \otimes \chi_{r'}$ for every integer $r' > 1$ dividing $r$, and as $f$ is nonmonomial, each of these twists is distinct. This gives $2^{s(q)}$ different newforms with the same Laplacian eigenvalue.
\end{proof}

\subsection{Str\"{o}mberg's Results for \texorpdfstring{Maa\ss{}}{Maa\80\337} Forms on \texorpdfstring{$\Gamma_0(9) \backslash \Hb$}{\textGamma\9040\200(9)\textbackslash H}}

\hyperref[introthm]{Theorem \ref*{introthm}} was previously proved in the particular case $q = 9$ by Str\"{o}mberg \cite{Stromberg} via quite different methods. He gives an orthogonal decomposition of the space $\Ac_0^{\ast}\left(\Gamma_0(9);\lambda\right)$ with respect to the Petersson inner product on $\Gamma_0(9) \backslash \Hb$:
\[\Ac_0^{\ast}\left(\Gamma_0(9);\lambda\right) = \left(\Ac_0^{\ast}\left(\SL_2(\Z);\lambda\right) \otimes \chi_3\right) \oplus \left(\Ac_0^{\ast}\left(\Gamma_0(3);\lambda\right) \otimes \chi_3\right) \oplus \left.\Ac_0^{\ast}\left(\Gamma^3;\lambda\right)\right|_{V_3}.\]
Here $\Ac_0^{\ast}\left(\Gamma^3;\lambda\right)$ denotes the new space of Maa\ss{} cusp forms on $\Gamma^3 \backslash \Hb$ having trivial congruence character and Laplacian eigenvalue $\lambda$, which consists of elements orthogonal to direct embeddings in $\Ac_0\left(\Gamma^3\right)$ of cusp forms in $\Ac_0\left(\SL_2(\Z)\right)$. The congruence subgroup $\Gamma^3$ is
\[\Gamma^3 \defeq \left\{\begin{pmatrix} a & b \\ c & d \end{pmatrix} \in \PSL_2(\Z) : ab + cd \equiv 0 \hspace{-.2cm} \pmod{3}\right\},\]
and $V_3$ is the map $\left. f\right|_{V_3}(z) \defeq f(3z)$, which lifts Maa\ss{} forms $f$ on $\Gamma^3 \backslash \Hb$ to Maa\ss{} forms $\left. f\right|_{V_3}$ on $\Gamma_0(9) \backslash \Hb$.

Letting $\lambda = \lambda_f$ denote the Laplacian eigenvalue of a Maa\ss{} cusp form $f$, Str\"{o}mberg uses the Selberg trace formula to prove a Weyl law of the form
\begin{equation}\label{Gamma3ratio}
\frac{\#\left\{f \in \Ac_0^{\ast}\left(\Gamma^3\right) : \left. f\right|_{V_3} \in \BB_0^{\ast}\left(\Gamma_0(9)\right), \ \lambda_f \leq T^2\right\}}{\#\left\{f \in \BB_0^{\ast}\left(\Gamma_0(9)\right) : \lambda_f \leq T^2\right\}} = \frac{2}{5} + o(1) \quad \text{as $T \to \infty$,}
\end{equation}
so that two-fifths of the newforms in $\BB_0^{\ast}\left(\Gamma_0(9)\right)$ arise from members of $\Ac_0^{\ast}\left(\Gamma^3\right)$.

Str\"{o}mberg then shows that the new space of Maa\ss{} cusp forms on $\Gamma^3 \backslash \Hb$ of trivial congruence character further decomposes into two orthogonal eigenspaces corresponding to the two possible eigenvalues $e^{2\pi i/3}$ and $e^{-2\pi i/3}$ of a nontrivial cusp form $f \in \Ac_0^{\ast}\left(\Gamma^3;\lambda\right)$ with respect to the action of the matrix $T \defeq \left(\begin{smallmatrix} 1 & 1 \\ 0 & 1 \end{smallmatrix}\right)$. Moreover, one can map each of these spaces to the other via the action of $J \defeq \left(\begin{smallmatrix} -1 & 0 \\ 0 & 1 \end{smallmatrix}\right) \in \PGL_2(\R)$ (where the action of a matrix with negative determinant means that we replace $z$ by $\overline{z}$), and this mapping preserves the Laplacian eigenvalue. So for every cusp form $f \in \Ac_0^{\ast}\left(\Gamma^3\right)$ for which $\left. f\right|_{V_3} \in \BB_0^{\ast}\left(\Gamma_0(9)\right)$, there exists a distinct cusp form $g \in \Ac_0^{\ast}\left(\Gamma^3\right)$ with $\left. g\right|_{V_3} \in \BB_0^{\ast}\left(\Gamma_0(9)\right)$ for which $\lambda_f = \lambda_g$. By \eqref{Gamma3ratio}, Str\"{o}mberg concludes the following.

\begin{proposition}[Str\"{o}mberg {\cite[Proposition 1.3]{Stromberg}}]
At least two-fifths of the newforms in $\BB_0^{\ast}\left(\Gamma_0(9)\right)$ have repeated Laplacian eigenvalues.
\end{proposition}

In \cite[Section 8]{Stromberg}, Str\"{o}mberg briefly discusses similar results for $\BB_0^{\ast}\left(\Gamma_0(25)\right)$, but already for $\BB_0^{\ast}\left(\Gamma_0(49)\right)$ his approach runs into new obstacles, where spectral multiplicity occurs but cannot immediately be explained by such Maa\ss{} cusp forms arising from congruence subgroups of lower level akin to $\Gamma^3$. In general, it seems that this method is unwieldy for showing spectral multiplicity for higher non-squarefree levels.

Nevertheless, in \cite[Section 6]{Stromberg}, Str\"{o}mberg shows that the two orthogonal parts of $\BB_0^{\ast}\left(\Gamma^3;\lambda\right)$, after embedding in $\BB_0^{\ast}\left(\Gamma_0(9);\lambda\right)$, are related by twisting by the primitive quadratic Dirichlet character $\chi_3$ modulo $3$. From this, it becomes clear that the key to forcing spectral multiplicity is not the existence of Maa\ss{} newforms arising from congruence subgroups of lower level such as $\Gamma^3$, but rather a description of $\BB_0^{\ast}\left(\Gamma_0(q);\lambda\right) \otimes \chi_r$ for each $r$ dividing $q$ for which there exists a primitive quadratic Dirichlet character $\chi_r$. It is this path that we take in strengthening Str\"{o}mberg's result to all moduli $q$ with an odd prime divisor $p$ for which $p^2$ divides $q$.

\section{Some Background on Automorphic Representations}\label{backgroundsect}

\subsection{Local and Global Definitions}

We recall some background information on cuspidal automorphic representations of $\GL_2$; see \cite{Bump,Gelbart,Goldfeld,Jacquet} for further details.

Let $F$ be an algebraic number field and let $\A_F$ denote the ring of ad\`{e}les of $F$. A place of $F$ will be denoted by $v$, and $F_v$ will denote the corresponding local field. We let $S_{\infty}$ and $S_f$ denote the set of archimedean and nonarchimedean places of $F$ respectively. A nonarchimedean place $v$ corresponds to a prime ideal $\pp$ of $F$ via the identification $\pp = \pp_v \cap \OO_F$, where $\pp_v$ is the maximal ideal of the ring of integers $\OO_v$ of the local field $F_v$. 

Given a character $\omega_v : F_v^{\times} \to \Cb^{\times}$ of $F_v^{\times}$ with $v$ nonarchimedean, we may write $\omega_v = \beta_v |\cdot|_v^s$ for some character $\beta_v$ of $\OO_v^{\times}$ (or rather, some character $\beta_v$ of $F_v^{\times}$ that is trivial on $\{\varpi_v^k : k \in \Z\}$) and some $s \in \Cb$. Here the absolute value $|\cdot|_v$ on $F_v^{\times}$ is normalised such that for a uniformiser $\varpi_v$ of $\OO_v$, so that $\varpi_v \OO_v = \pp_v$, we have that $|\varpi_v|_v^{-1} = \# \OO_v / \pp_v \eqdef q_v$, the cardinality of the residue field of $F_v$. If $\omega_v$ is trivial on $\OO_v^{\times}$, then $\omega_v$ is said to be unramified and of conductor $\OO_v$ and conductor exponent $c(\omega_v) = 0$. Otherwise, $\omega_v$ is said to be ramified with conductor $\pp_v^m$ and conductor exponent $c\left(\omega_v\right) = m$ with $m$ the minimal positive integer for which $\omega_v\left(1 + \pp_v^m\right) = 1$.

Let $\left(\pi_v,V_v\right)$ be a generic irreducible admissible representation of $\GL_2(F_v)$ with $v$ a nonarchimedean place of $F$. We say that $\pi_v$ is unramified with conductor $\OO_v$ and conductor exponent $c(\pi_v) = 0$ if the vector subspace
\[V_v^{\GL_2\left(\OO_v\right)} \defeq \left\{\xi_v \in V_v : \pi_v\left(g_v\right) \cdot \xi_v = \xi_v \text{ for all $g_v \in \GL_2\left(\OO_v\right)$}\right\}\]
of $V_v$ is nontrivial, in which case it must be one-dimensional. Otherwise, $\pi_v$ is said to be ramified with conductor $\pp_v^m$ and conductor exponent $c(\pi_v) = m$ with $m$ the minimal positive integer for which
\[V_v^{K_0\left(\pp_v^m\right),\omega_{\pi_v}} \defeq \left\{\xi_v \in V_v : \pi_v\begin{pmatrix} a & b \\ c & d \end{pmatrix} \cdot \xi_v = \omega_{\pi_v}(d) \xi_v \text{ for all $\begin{pmatrix} a & b \\ c & d \end{pmatrix} \in K_0\left(\pp_v^m\right)$}\right\}\]
is nontrivial; here $\omega_{\pi_v}$ is the central character of $\pi_v$ and
\[K_0\left(\pp_v^m\right) \defeq \left\{\begin{pmatrix} a & b \\ c & d \end{pmatrix} \in \GL_2\left(\OO_v\right) : c \in \pp_v^m\right\}.\]
Via the work of Casselman \cite{Casselman}, $V_v^{K_0\left(\pp_v^m\right),\omega_{\pi_v}}$ is one-dimensional, and so contains a nonzero element $\xi_v^{\circ} \in V_v$ named the local newvector of $\pi_v$ that is unique up to scalar multiplication.

When $v$ is archimedean, the analogue of a local newform is a local lowest weight vector $\xi_v^{\circ}$ of $V_v$; these are similarly unique up to scalar multiplication.

Let $(\pi,V)$ be a cuspidal automorphic representation of $\GL_2(\A_F)$ of central character $\omega_{\pi} = \prod_v \omega_{\pi_v}$ on a space $V$ of cuspidal automorphic forms of $\GL_2(\A_F)$ of central character $\omega_{\pi}$. Throughout, we will always assume that automorphic representations and Hecke characters are unitary. By Flath's tensor product theorem (see \cite[Theorem 3.3.3]{Bump}), there exists an irreducible admissible $\left(\gl_2\left(F_v\right)_{\Cb},K_v\right)$-module $\left(\pi_v,V_v\right)$ for each archimedean place $v$ and a generic irreducible admissible unitarisable representation $\left(\pi_v, V_v\right)$ of $\GL_2(F_v)$ with central character $\omega_{\pi_v}$ for each nonarchimedean place $v$ such that $(\pi,V)$ is isomorphic to the restricted tensor product $\left(\bigotimes_v \pi_v, \bigotimes_v V_v\right)$, where we may make the identification according to the local newvector or local lowest weight vector $\xi_v^{\circ} \in V_v$ at each place. The global newvector of $\pi$ is the cuspidal automorphic form $\xi^{\circ} \in V$ that corresponds to the pure tensor $\bigotimes_v \xi_v^{\circ}$.

The arithmetic conductor of $\pi$ is defined to be the integral ideal
\[\qq \defeq \prod_{v \in S_f} \left(\pp_v \cap \OO_F\right)^{c\left(\pi_v\right)}\]
of $\OO_F$; the representation $\pi$ is ramified at only finitely many places, so this is well-defined. Similarly, the arithmetic conductor of a Hecke character $\omega = \prod_v \omega_v$ of $F^{\times} \backslash \A_F^{\times}$ is
\[\ff \defeq \prod_{v \in S_f} \left(\pp_v \cap \OO_F\right)^{c\left(\omega_v\right)}.\]
When $\omega = \omega_{\pi}$ is the central character of $\pi$, $c\left(\omega_{\pi_v}\right)$ is at most $c(\pi_v)$ for each nonarchimedean place $v$ of $F$, so that the arithmetic conductor $\ff$ of $\omega_{\pi}$ divides the arithmetic conductor $\qq$ of $\pi$.

When $F = \Q$, the global newvector of $\pi$ is the ad\`{e}lic lift of a classical newform $f$; see \cite[Section 4.12]{Goldfeld}. Moreover, the central character $\omega_{\pi}$ is the id\`{e}lic lift of the character $\chi$ of $f$, while the arithmetic conductor of $\omega_{\pi}$ is the ideal in $\Z$ generated by the conductor of $\chi$, the arithmetic conductor of $\pi$ is the ideal generated by the level $q$ of $f$, and the archimedean component $\pi_{\infty}$ of $\pi$ specifies its Laplacian eigenvalue $\lambda_f$ and its weight $k$.

Conversely, given a newform $f$ of level $q$ and character $\chi$, there exists a cuspidal automorphic representation of $\GL_2(\A_{\Q})$, unique up to isomorphism, with conductor exponent $c(\pi_p)$ at each prime $p$ satisfying $p^{c(\pi_p)} \parallel q$, whose central character $\omega_{\pi}$ is the id\`{e}lic lift of $\chi$, and whose global newvector $\xi^{\circ}$ is the ad\`{e}lic lift of $f$. In particular, there is a bijective correspondence between newforms and cuspidal automorphic representations.

\subsection{Bounding the Conductor Exponent of a Twist}

Let $(\pi,V)$ be a cuspidal automorphic representation of $\GL_2(\A_F)$ with central character $\omega_{\pi}$. The twist of $\pi$ by a Hecke character $\omega$ is the cuspidal automorphic representation $(\pi \otimes \omega,V_{\omega})$ with central character $\omega^2 \omega_{\pi}$ acting on the vector space
\[V_{\omega} \defeq \left\{\phi_{\omega} : \phi \in V\right\},\]
where $\phi_{\omega}(g) \defeq \omega(\det g) \phi(g)$ for $g \in \GL_2(\A_F)$. If we write $\omega = \prod_v \omega_v$, then the local components of $\left(\pi \otimes \omega,V_{\omega}\right)$ are again generic irreducible admissible representations $\left(\pi_v \otimes \omega_v,V_v\right)$ with central character $\omega_v^2 \omega_{\pi_v}$ and action
\[\left(\pi_v \otimes \omega_v\right)\left(g_v\right) \cdot \xi_v \defeq \omega_v(\det g_v) \pi_v\left(g_v\right) \cdot \xi_v\]
for $g_v \in \GL_2\left(F_v\right)$ and $\xi_v \in V_v$.

Let $\pi$ be a cuspidal automorphic representation of $\GL_2(\A_F)$, and suppose that there exists a twist $\pi \otimes \omega$ of $\pi$ by some Hecke character $\omega$ such that $\pi \otimes \omega$ has the same central character and arithmetic conductor as $\pi$. Then $\omega$ must be a quadratic Hecke character, and the local conductor exponents must satisfy $c(\pi_v \otimes \omega_v) = c(\pi_v)$ at every nonarchimedean place $v$. This equality automatically holds when $\omega_v$ is unramified. If $\omega_v$ is ramified, on the other hand, then this equality is not always ensured. We will determine when this occurs in the particular case when the central character of $\pi$ is trivial. The result is obtained on a case-by-case analysis, which requires the classification of generic irreducible admissible representations of $\GL_2(F_v)$. A key tool is the following bound for the conductor exponent.

\begin{lemma}\label{twistcondlemma}
Let $v$ be a nonarchimedean place of an algebraic number field $F$. Let $(\pi_v,V_v)$ be a generic irreducible admissible representation of $\GL_2(F_v)$ with central character $\omega_{\pi_v}$, and let $\omega_v$ be a unitary character of $F_v^{\times}$. Then we have the following bound for the conductor exponent of $\pi_v \otimes \omega_v$:
\[c\left(\pi_v \otimes \omega_v\right) \leq \max\left\{c\left(\pi_v\right), c\left(\omega_v\right) + c\left(\omega_{\pi_v}\right), 2c\left(\omega_v\right)\right\}.\]
\end{lemma}

This is proved in \cite[Proposition 3.1]{Atkin2} by Atkin and Li in the classical setting of holomorphic newforms (that is, cuspidal automorphic representations of $\GL_2(\A_{\Q})$ whose archimedean component is a discrete series representation or a limit of discrete series representation). We give a local proof here that essentially mimics Atkin and Li's proof.

\begin{proof}
We must show that that there exists a nonzero vector $\xi_v' \in V_v$ satisfying
\begin{equation}\label{piomegatwisteq}
\left(\pi_v \otimes \omega_v\right) \begin{pmatrix} a & b \\ c & d \end{pmatrix} \cdot \xi_v' = \omega_v(d)^2 \omega_{\pi_v}(d) \xi_v'
\end{equation}
for all $\left(\begin{smallmatrix} a & b \\ c & d \end{smallmatrix}\right) \in K_0\left(\pp_v^m\right)$, where
\[m = \max\left\{c\left(\pi_v\right), c\left(\omega_v\right) + c\left(\omega_{\pi_v}\right), 2c\left(\omega_v\right)\right\}.\]
As $\pi_v$ is generic, there exists a nontrivial continuous unramified additive character $\psi_v$ and a Whittaker functional $\Lambda_v : V_v \to \Cb$ satisfying
\[\Lambda_v\left(\pi_v\begin{pmatrix} 1 & x \\ 0 & 1 \end{pmatrix} \cdot \xi_v\right) = \psi_v(x) \Lambda_v\left(\xi_v\right)\]
for every $x \in F_v$ and $\xi_v \in V_v$. Let $\xi_v^{\circ} \in V_v^{K_0(\pp_v^{c(\pi_v)}), \omega_{\pi_v}}$ be the local newvector of $\pi_v$ satisfying $\Lambda_v(\xi_v^{\circ}) = 1$. We define $\xi_v' \in V_v$ by
\[\xi_v' \defeq \int\limits_{\varpi_v^{-c\left(\omega_v\right)} \OO_v^{\times}} \omega_v(x) \pi_v\begin{pmatrix} 1 & x \\ 0 & 1 \end{pmatrix} \cdot \xi_v^{\circ} \, dx,\]
where the Haar measure $dx$ is normalised to give $\OO_v$ volume $1$. As $\Lambda_v\left(\xi_v^{\circ}\right) = 1$, we have that $\Lambda_v\left(\xi_v'\right) = \epsilon_v\left(0,\omega_v,\psi_v\right)$, where
\[\epsilon_v\left(s,\omega_v,\psi_v\right) \defeq \int\limits_{\varpi_v^{-c\left(\omega_v\right)} \OO_v^{\times}} |x|^{-s} \omega_v(x) \psi_v(x) \, dx\]
is the local epsilon factor of $\omega_v$, which is nonzero for all $s \in \Cb$. It follows that $\xi_v'$ is nonzero.

For $\left(\begin{smallmatrix} a & b \\ c & d \end{smallmatrix}\right) \in K_0(\pp_v^m)$, so that $a,d \in \OO_v^{\times}$ and $c \in \pp_v^m$, we have that
\begin{equation}\label{piomegatwistinteq}
\left(\pi_v \otimes \omega_v\right) \begin{pmatrix} a & b \\ c & d \end{pmatrix} \cdot \xi_v' = \omega_v(ad - bc) \int\limits_{\varpi_v^{-c\left(\omega_v\right)} \OO_v^{\times}} \omega_v(x) \pi_v\left(\begin{pmatrix} a & b \\ c & d \end{pmatrix} \begin{pmatrix} 1 & x \\ 0 & 1 \end{pmatrix}\right) \cdot \xi_v^{\circ} \, dx.
\end{equation}
As $\omega_v(ad - bc) = \omega_v(ad)$ due to the fact that $m \geq c(\omega_v)$ and
\[\begin{pmatrix} a & b \\ c & d \end{pmatrix} \begin{pmatrix} 1 & x \\ 0 & 1 \end{pmatrix} = \begin{pmatrix} a & ax + b \\ c & cx + d \end{pmatrix} = \begin{pmatrix} 1 & ad^{-1} x \\ 0 & 1 \end{pmatrix} \begin{pmatrix} a(1 - cd^{-1}x) & b - acd^{-1}x^2 \\ c & cx + d \end{pmatrix},\]
where the second matrix on the right-hand side is in $K_0(\pp_v^m)$ due to the fact that $m \geq 2c(\omega_v)$, \eqref{piomegatwistinteq} is equal to
\[\omega_v(ad) \int\limits_{\varpi_v^{-c\left(\omega_v\right)} \OO_v^{\times}} \omega_v(x) \omega_{\pi_v}(cx + d) \pi_v\begin{pmatrix} 1 & ad^{-1} x \\ 0 & 1 \end{pmatrix} \cdot \xi_v^{\circ} \, dx\]
because $m \geq c(\pi_v)$ and so $V_v^{K_0(\pp_v^m), \omega_{\pi_v}} \supset V_v^{K_0(\pp_v^{c(\pi_v)}), \omega_{\pi_v}} \ni \xi_v^{\circ}$. Making the change of variables $x \mapsto a^{-1}dx$, \eqref{piomegatwistinteq} becomes
\[\omega_v(d)^2 \omega_{\pi_v}(d) \int\limits_{\varpi_v^{-c\left(\omega_v\right)} \OO_v^{\times}} \omega_v(x) \omega_{\pi_v}(a^{-1}cx + 1) \pi_v\begin{pmatrix} 1 & x \\ 0 & 1 \end{pmatrix} \cdot \xi_v^{\circ} \, dx.\]
As $m \geq c(\omega_v) + c(\omega_{\pi_v})$, we have that $\omega_{\pi_v}(a^{-1}cx + 1) = 1$. It follows that \eqref{piomegatwisteq} holds, as desired.
\end{proof}

\subsection{Representations of \texorpdfstring{$\GL_2(F_v)$}{GL\9040\202(F\9035\145)} at Nonarchimedean Places}

Let $v$ be a nondyadic nonarchimedean place of $F$, which is to say that $\pp_v \cap \Z \neq 2\Z$. There exists a unique nontrivial character $\beta_v$ of $\OO_v^{\times}$ satisfying $\beta_v^2 = 1$, which we denote $\beta_{v,\Quad}$. We may view $\beta_{v,\Quad}$ as a character of $F_v^{\times}$ that is trivial on $\{\varpi_v^k : k \in \Z\}$; as $v$ is nondyadic, $\beta_{v,\Quad}$ has conductor $\pp_v$ (and hence conductor exponent $1$). As every character $\omega_v$ of $F_v^{\times}$ is equal to $\beta_v |\cdot|_v^s$ for some character $\beta_v$ of $\OO_v^{\times}$ and some $s \in \Cb$, it follows that there are three quadratic characters of $F_v^{\times}$, namely
\[|\cdot|_v^{\frac{\pi i}{\log q_v}}, \qquad \beta_{v,\Quad}, \qquad \beta_{v,\Quad} |\cdot|_v^{\frac{\pi i}{\log q_v}},\]
where as before $q_v \defeq \# \OO_v / \pp_v = |\varpi_v|_v^{-1}$. The first character is unramified, and so twisting a representation $\pi_v$ by $|\cdot|_v^{\frac{\pi i}{\log q_v}}$ does not change the conductor of $\pi_v$, while for the third, we have that
\[\pi_v \otimes \beta_{v,\Quad} |\cdot|_v^{\frac{\pi i}{\log q_v}} = \left(\pi_v \otimes |\cdot|_v^{\frac{\pi i}{\log q_v}}\right) \otimes \beta_{v,\Quad}.\]
So it suffices to classify the representations $\pi_v$ for which $c(\pi_v \otimes \beta_{v,\Quad}) = c(\pi_v)$.

Let $\pi_v$ be a generic irreducible admissible unitarisable representation of $\GL_2(F_v)$ with central character $\omega_v$. We recall that such representations can be classified as either principal series representations, special representations, or supercuspidal representations. Standard reference for the properties of these representations are \cite[Chapter 4]{Bump} and \cite[Chapter 6]{Goldfeld}, while \cite{Schmidt} discusses the conductor exponents of these representations.

\subsubsection{Principal Series Representations of \texorpdfstring{$\GL_2(F_v)$}{GL\9040\202(F\9035\145)}}

A principal series representation $\pi_v$ of $\GL_2(F_v)$ is unitarily induced from a representation of the Borel subgroup of $\GL_2(F_v)$, and these representations are in turn indexed by two characters
\[\omega_{v,1} = \beta_{v,1} |\cdot|_v^{s_1}, \qquad \omega_{v,2} = \beta_{v,2} |\cdot|_v^{s_2}\]
of $F_v^{\times}$; here $\beta_{v,1},\beta_{v,2}$ are characters of $\OO_v^{\times}$ and $s_1,s_2 \in \Cb$. We write
\[\pi_v \cong \omega_{v,1} \boxplus \omega_{v,2}.\]
This representation is irreducible and unitarisable if and only if either $s_1, s_2 \in i\R$, or $s_1 + s_2 \in i\R$ with $s_1 - s_2 \in (-1,1)$ and $\beta_{v,1} = \beta_{v,2}$. The central character of $\pi_v$ is
\[\omega_{\pi_v} = \omega_{v,1} \omega_{v,2} = \beta_{v,1} \beta_{v,2} |\cdot|_v^{s_1 + s_2}.\]
The conductor exponent $c(\pi_v)$ of $\pi_v$ is $c(\omega_{v,1}) + c(\omega_{v,2})$. If $\omega_v'$ is a character of $F_v^{\times}$, the twist of $\pi_v$ by $\omega_v'$ is the principal series representation
\[\pi_v \otimes \omega_v' \cong \omega_{v,1} \omega_v' \boxplus \omega_{v,2} \omega_v'.\]
If $\pi_v$ has trivial central character, then $\omega_{v,2} = \omega_{v,1}^{-1}$, so that $\beta_{v,2} = \beta_{v,1}^{-1}$ and $s_2 = -s_1$, and $c(\pi_v) = 2c(\omega_{v,1})$.

\begin{lemma}\label{principaltwist}
Let $\pi_v$ be a ramified irreducible unitarisable principal series representation with trivial central character, so that
\[\pi_v \cong \beta_v |\cdot|_v^s \boxplus \beta_v^{-1} |\cdot|_v^{-s}\]
for some nontrivial character $\beta_v$ of $\OO_v^{\times}$ and some $s \in i\R \cup (-1/2,1/2)$. Then $c(\pi_v \otimes \beta_{v,\Quad}) = c(\pi_v)$ unless $\beta_v = \beta_{v,\Quad}$, in which case $c(\pi_v) = 2$ but $c(\pi_v \otimes \beta_{v,\Quad}) = 0$.
\end{lemma}

\begin{proof}
The conductor exponent $c(\pi_v)$ of $\pi_v$ is $2c(\beta_v)$, and in particular is positive and even. We have that
\[\pi_v \otimes \beta_{v,\Quad} \cong \beta_v \beta_{v,\Quad} |\cdot|_v^s \boxplus \beta_v^{-1} \beta_{v,\Quad} |\cdot|_v^{-s},\]
which also has trivial central character. If $\beta_v = \beta_{v,\Quad}$, in which case $c(\pi_v) = 2$, then
\[\pi_v \otimes \beta_{v,\Quad} = |\cdot|_v^s \boxplus |\cdot|_v^{-s},\]
and so $c(\pi_v \otimes \beta_{v,\Quad}) = 0$. If $\beta_v \neq \beta_{v,\Quad}$, then
\begin{align*}
c\left(\pi_v \otimes \beta_{v,\Quad}\right) & = c\left(\beta_v \beta_{v,\Quad}\right) + c\left(\beta_v^{-1} \beta_{v,\Quad}\right)	\\
& = c\left(\beta_v\right) + c\left(\beta_v^{-1}\right)	\\
& = c\left(\pi_v\right)
\end{align*}
as $c(\beta_v \beta_{v,\Quad}) \leq \max\{c(\beta_v), c(\beta_{v,\Quad})\} \leq 1$, and also $c(\beta_v \beta_{v,\Quad}) \geq 1$ as $\beta_v \beta_{v,\Quad}$ is a nontrivial character of $\OO_v^{\times}$.
%we prove by strong induction that $c(\pi_v \otimes \beta_{v,\Quad}) = c(\pi_v)$. The base case is $c(\pi_v) = 2$. Then by \hyperref[twistcondlemma]{Lemma \ref*{twistcondlemma}},
%\[c\left(\pi_v \otimes \beta_{v,\Quad}\right) \leq \max\left\{c\left(\pi_v\right), c\left(\beta_{v,\Quad}\right) + c\left(\omega_{\pi_v}\right), 2c\left(\beta_{v,\Quad}\right)\right\} = 2,\]
%but as $\pi_v \otimes \beta_{v,\Quad}$ has trivial central character and is not unramified, we have that $c(\pi_v \otimes \beta_{v,\Quad}) \geq 2$, and so equality holds. For the strong induction, we observe that again
%\[c\left(\pi_v \otimes \beta_{v,\Quad}\right) \leq \max\left\{c\left(\pi_v\right), c\left(\beta_{v,\Quad}\right) + c\left(\omega_{\pi_v}\right), 2c\left(\beta_{v,\Quad}\right)\right\} = c\left(\pi_v\right)\]
%by \hyperref[twistcondlemma]{Lemma \ref*{twistcondlemma}}, but that $c(\pi_v \otimes \beta_{v,\Quad}) \geq c(\pi_v)$ by the strong induction hypothesis, for otherwise we could twist by $\beta_{v,\Quad}$ once again. This yields the result.
\end{proof}

\subsubsection{Special Representations of \texorpdfstring{$\GL_2(F_v)$}{GL\9040\202(F\9035\145)}}

A special representation is a twist of a Steinberg represention: it is the unique irreducible subrepresentation
\[\pi_v \cong \omega_v \St_v\]
of codimension one of the reducible principal series representation $\omega_v |\cdot|_v^{1/2} \boxplus \omega_v |\cdot|_v^{-1/2}$, where $\omega_v = \beta_v |\cdot|_v^s$ is a character of $F_v^{\times}$, with $\beta_v$ a character of $\OO_v^{\times}$ and $s \in \Cb$; $\pi_v$ is unitary precisely when $s \in i\R$. The central character of $\pi_v$ is
\[\omega_{\pi_v} = \omega_v^2 = \beta_v^2 |\cdot|_v^{2s}.\]
The conductor exponent is
\begin{equation}\label{specialconductor}
c\left(\pi_v\right) = \begin{dcases*}
1 & if $c\left(\omega_v\right) = 0$,	\\
2c\left(\omega_v\right) & otherwise.
\end{dcases*}
\end{equation}
If $\omega_v'$ is a character of $F_v^{\times}$, the twist of $\pi_v$ by $\omega_v'$ is the special representation
\[\pi_v \otimes \omega_v' = \omega_v \omega_v' \St_v.\]
Suppose that $\pi_v$ is unitary and has trivial central character. Then $\omega_v^2 = 1$; consequently, we may assume without loss of generality that $s = s_{\Quad} \in \left\{0, \frac{\pi i}{\log q_v}\right\}$, and we must have that $\beta_v \in \left\{1, \beta_{v,\Quad}\right\}$ and that
\[c\left(\pi_v\right) = \begin{dcases*}
1 & if $\beta_v = 1$,	\\
2 & if $\beta_v = \beta_{v,\Quad}$.
\end{dcases*}\]

\begin{lemma}\label{specialtwist}
Let $\pi_v \cong \omega_v \St_v$ be a unitary special representation with trivial central character, so that $\omega_v = \beta_v |\cdot|_v^{s_{\Quad}}$ with $\beta_v \in \{1, \beta_{v,\Quad}\}$ and $s_{\Quad} \in \left\{0, \frac{\pi i}{\log q_v}\right\}$. Then
\[\pi_v \otimes \beta_{v,\Quad} \cong \begin{dcases*}
\beta_{v,\Quad} |\cdot|_v^{s_{\Quad}} \St_v & if $\beta_v = 1$,	\\
|\cdot|_v^{s_{\Quad}} \St_v & if $\beta_v = \beta_{v,\Quad}$,
\end{dcases*}\]
and hence
\[c\left(\pi_v \otimes \beta_{v,\Quad}\right) = \begin{dcases*}
c\left(\pi_v\right) + 1 & if $\beta_v = 1$,	\\
c\left(\pi_v\right) - 1 & if $\beta_v = \beta_{v,\Quad}$.
\end{dcases*}\]
\end{lemma}

\begin{proof}
This follows from \eqref{specialconductor} with $\pi_v$ replaced by $\pi_v \otimes \beta_{v,\Quad}$.
\end{proof}

\begin{lemma}\label{specialpv}
Every generic irreducible admissible unitarisable representation of $\GL_2(F_v)$ with conductor $\pp_v$ and trivial central character is of the form $|\cdot|_v^{s_{\Quad}} \St_v$ with $s_{\Quad} \in \left\{0, \frac{\pi i}{\log q_v}\right\}$.
\end{lemma}

\begin{proof}
Such a representation cannot be a principal series representation, as the central character is trivial but the conductor exponent is not even, or a supercuspidal representation, as the conductor exponent is less than $2$. It remains to note that every special representation of conductor $\pp_v$ and trivial central character $\omega_v^2$ is of the desired form.
\end{proof}

\subsubsection{Supercuspidal Representations of \texorpdfstring{$\GL_2(F_v)$}{GL\9040\202(F\9035\145)}}

A supercuspidal representation is the compact induction to $\GL_2(F_v)$ of a finite-dimensional representation $\rho_{\pi_v}$ of a maximal open subgroup $H$ of $\GL_2(F_v)$ such that $H$ is compact modulo the centre $Z(F_v)$ of $\GL_2(F_v)$. Every maximal open subgroup of $\GL_2(F_v)$ that is compact modulo the centre is conjugate to either $Z(F_v) \GL_2(\OO_v)$ or $N K_0(\pp_v)$, the normaliser of $K_0(\pp_v)$ in $\GL_2(F_v)$. We say that a supercuspidal representation $\pi_v$ is of type I if $H$ is conjugate to $Z(F_v) \GL_2(\OO_v)$ and of type II if $H$ is conjugate to $N K_0(\pp_v)$. Supercuspidal representations always have conductor exponent $c(\pi_v)$ at least $2$. The twist $\pi_v \otimes \omega_v'$ of a supercuspidal representation $\pi_v$ by a character $\omega_v'$ of $F_v^{\times}$ is also a supercuspidal representation.

\begin{lemma}\label{supercuspidaltwist}
Let $\pi_v$ be a supercuspidal representation with trivial central character. Then the twist $\pi_v \otimes \beta_{v,\Quad}$ is a supercuspidal representation with trivial central character and conductor exponent $c(\pi_v \otimes \beta_{v,\Quad}) = c(\pi_v)$.
\end{lemma}

\begin{proof}
It is clear that $\pi_v \otimes \beta_{v,\Quad}$ is a supercuspidal representation with trivial central character. The fact that the conductor remains unchanged follows from \hyperref[twistcondlemma]{Lemma \ref*{twistcondlemma}}: we have that
\begin{align*}
c\left(\pi_v\right) & = c\left(\left(\pi_v \otimes \beta_{v,\Quad}\right) \otimes \beta_{v,\Quad}\right)	\\
& \leq \max\left\{c\left(\pi_v \otimes \beta_{v,\Quad}\right), c\left(\beta_{v,\Quad}\right) + c(1), 2c\left(\beta_{v,\Quad}\right)\right\}	\\
& = c\left(\pi_v \otimes \beta_{v,\Quad}\right)	\\
& \leq \max\left\{c\left(\pi_v\right), c\left(\beta_{v,\Quad}\right) + c(1), 2c\left(\beta_{v,\Quad}\right)\right\}	\\
& = c\left(\pi_v\right).
\qedhere
\end{align*}
\end{proof}

\subsection{\texorpdfstring{$\left(\gl_2\left(F_v\right)_{\Cb},K_v\right)$}{(gl\9040\202(F\9035\145)C,K\9035\145)}-Modules}

Let $S_{\R}$ and $S_{\Cb}$ denote the real and complex places of $F$ respectively. For $v \in S_{\infty} = S_{\R} \cup S_{\Cb}$, a generic irreducible admissible unitarisable $(\gl_2(F_v)_{\Cb}, K_v)$-module $\pi_v$ corresponds to an infinitesimal equivalence class of principal series representations or discrete series representations, with the latter only possible if $v \in S_{\R}$.

\subsubsection{Principal Series Representations of \texorpdfstring{$\GL_2(\R)$}{GL\9040\202(R)}}

Let $v$ be a real place of $F$. Unitarisable principal series representations of $\GL_2(F_v)$ are of the form
\[\pi_v \cong \omega_{v,1} \boxplus \omega_{v,2}\]
for some characters $\omega_{v,1} = \sgn_v^{m_1} |\cdot|_v^{s_1}$, $\omega_{v,2} = \sgn_v^{m_2} |\cdot|_v^{s_2}$ of $\R^{\times}$, with $m_1, m_2 \in \Z/2\Z$ and either $s_1, s_2 \in i\R$, or $s_1 + s_2 \in i\R$ with $s_1 - s_2 \in (-1,1)$ and $m_1 = m_2$. Here $\sgn_v(x) \defeq x / |x|_v$ is the sign of an element $x \in F_v$, with $|x|_v \defeq \max\{x,-x\}$ the usual absolute value on $\R$. The central character of $\pi_v$ is
\[\omega_{\pi_v} = \omega_{v,1} \omega_{v,2} = \sgn_v^{m_1 + m_2} |\cdot|_v^{s_1 + s_2}.\]
If $\omega_v'$ is a character of $\R^{\times}$, the twist of $\pi_v$ by $\omega_v'$ is the principal series representation
\[\pi_v \otimes \omega_v' \cong \omega_{v,1} \omega_v' \boxplus \omega_{v,2} \omega_v'.\]
We define the weight $k_v \in \{0,1\}$ and spectral parameter $s_v \in i\R \cup (-1/2,1/2)$ of $\pi_v$ to be
\[k_v \defeq |m_1 - m_2|, \qquad s_v \defeq \frac{s_1 - s_2}{2};\]
note that the latter is defined up to multiplication by $\pm 1$ since $\omega_{v,1} \boxplus \omega_{v,2} \cong \omega_{v,2} \boxplus \omega_{v,1}$. The quadratic twist of $\pi_v$ by the unique quadratic character $\sgn_v$ of $\R^{\times}$ merely sends $(m_1,m_2)$ to $(m_1 + 1, m_2 + 1)$ in $\Z/2\Z$ while leaving $s_1,s_2$ unaltered. Thus the weight $k_v$ and spectral parameter $s_v$ remain unchanged. When $\pi_v$ has trivial central character, we have that $\omega_{v,2} = \omega_{v,1}^{-1}$, so that $m_2 = m_1$ and $s_2 = -s_1$. In this case, we have that $k_v = 0$ and $s_v = s_1 \in i\R \cup (-1/2,1/2)$.

\subsubsection{Discrete Series Representations of \texorpdfstring{$\GL_2(\R)$}{GL\9040\202(R)}}

A unitarisable discrete series representation of $\GL_2(F_v)$, $v \in S_{\R}$, is the unique irreducible subrepresentation
\[\pi_v \cong \DD\left(\omega_{v,1}, \omega_{v,2}\right)\]
of codimension one of the reducible principal series representation $\omega_{v,1} \boxplus \omega_{v,2}$ for some characters $\omega_{v,1} = \sgn_v^{m_1} |\cdot|_v^{s_1}$, $\omega_{v,2} = \sgn_v^{m_2} |\cdot|_v^{s_2}$ of $\R^{\times}$ with $m_1, m_2 \in \Z/2\Z$, and $s_1, s_2 \in \Cb$ such that $s_1 + s_2 \in i\R$, $s_1 - s_2 \in \Z \setminus \{0\}$, and $s_1 - s_2 + 1 \equiv m_1 - m_2 \pmod{2}$. The representation $\DD(\omega_{v,1}, \omega_{v,2})$ is isomorphic to $\DD(\omega_{v,2}, \omega_{v,1})$, $\DD(\sgn_v \omega_{v,1}, \sgn_v \omega_{v,2})$, and $\DD(\sgn_v \omega_{v,2}, \sgn_v \omega_{v,1})$. The central character of $\pi_v$ is again
\[\omega_{\pi_v} = \omega_{v,1} \omega_{v,2} = \sgn_v^{m_1 + m_2} |\cdot|_v^{s_1 + s_2}.\]
If $\omega_v'$ is a unitary character of $\R^{\times}$, the twist of $\pi_v$ by $\omega_v'$ is the discrete series representation
\[\pi_v \otimes \omega_v' \cong \DD\left(\omega_{v,1} \omega_v', \omega_{v,2} \omega_v'\right).\]
We define the weight $k_v \in \N \setminus \{1\}$ and spectral parameter $s_v$ of $\pi_v$ to be
\[k_v \defeq |s_1 - s_2| + 1, \qquad s_v \defeq \frac{s_1 - s_2}{2},\]
with the latter defined up to multiplication by $\pm 1$; note that the weight of a discrete series representation is defined differently to the weight of a principal series representation. The quadratic twist of $\pi_v$ by $\sgn_v$ is isomorphic to $\pi_v$, and in particular leaves the weight $k_v$ and spectral parameter $s_v$ unchanged. Again, when $\pi_v$ has trivial central character, we have that $\omega_{v,2} = \omega_{v,1}^{-1}$, so that $m_2 = m_1$ and $s_2 = -s_1$. In this case, we have that $k_v = 2|s_1| + 1 \in 2\N$ and $s_v = s_1 \in 1/2 + \Z$.

\subsubsection{Principal Series Representations of \texorpdfstring{$\GL_2(\Cb)$}{GL\9040\202(C)}}

A unitarisable principal series representation of $\GL_2(F_v)$ with $v \in S_{\Cb}$ is of the form $\pi_v \cong \omega_{v,1} \boxplus \omega_{v,2}$ for some characters $\omega_{v,1} = e^{i m_1 \arg_v} |\cdot|_v^{s_1}$, $\omega_{v,2} = e^{i m_2 \arg_v} |\cdot|_v^{s_2}$ of $\Cb^{\times}$, with $m_1, m_2 \in \Z$ and either $s_1, s_2 \in i\R$, or $s_1 + s_2 \in i\R$ with $s_1 - s_2 \in (-1,1)$ and $m_1 = m_2$. Here $\arg_v(z)$ is the argument of an element $z \in F_v$, so that $e^{i \arg_v(z)} \defeq z / |z|_v^{1/2}$ with $|z|_v \defeq z \overline{z}$ the square of the usual absolute value on $\Cb$. If $\omega_v'$ is a character of $\R^{\times}$, the twist of $\pi_v$ by $\omega_v'$ is the principal series representation
\[\pi_v \otimes \omega_v' \cong \omega_{v,1} \omega_v' \boxplus \omega_{v,2} \omega_v'.\]
We again define the weight $k_v \in \N \cup \{0\}$ and spectral parameter $s_v \in i\R \cup (-1/2,1/2)$ of $\pi_v$ to be
\[k_v \defeq |m_1 - m_2|, \qquad s_v \defeq \frac{s_1 - s_2}{2},\]
with the latter again defined up to multiplication by $\pm 1$ as $\omega_{v,1} \boxplus \omega_{v,2} \cong \omega_{v,2} \boxplus \omega_{v,1}$. Note that there is no quadratic character of $\Cb^{\times}$. When $\pi_v$ has trivial central character, we have that $\omega_{v,2} = \omega_{v,1}^{-1}$, so that $m_2 = -m_1$ and $s_2 = -s_1$. In this case, we have that $k_v = 2|m_1| \in 2\N \cup \{0\}$ and $s_v = s_1 \in i\R \cup (-1/2,1/2)$.

\section{Proof of \texorpdfstring{\hyperref[twistthmclassical]{Theorem \ref*{twistthmclassical}}}{Theorem \ref{twistthmclassical}}}\label{prooftwistthmsect}

A cuspidal automorphic representation $\pi$ of $\GL_2(\A_F)$ is said to have archimedean weight and spectral parameter vectors $(\kbf,\sbf) = (k_v,s_v)_{v \in S_{\infty}}$ if at each archimedean place $v$ of $F$, the local component $\pi_v$ of $\pi$ has weight $k_v$ and spectral parameter $s_v$. We let $\Xx(K_0(\qq),\kbf,\sbf)$ denote the set of isomorphism classes of cuspidal automorphic representations of $\GL_2(\A_F)$ of arithmetic conductor $\qq$, trivial central character, and archimedean weight and spectral parameter vectors $(\kbf,\sbf)$, while $\Xx(K_0(\qq),\kbf,\sbf) \otimes \omega_{\pp} \defeq \left\{\pi \otimes \omega_{\pp} : \pi \in \Xx(K_0(\qq),\kbf,\sbf)\right\}$.

\begin{theorem}\label{twistthmadelic}
Let $\qq$ be an integral ideal of $\OO_F$ divisible by a nondyadic prime ideal $\pp$. Let $\omega_{\pp}$ denote a quadratic Hecke character of $F^{\times} \backslash \A_F^{\times}$ that is ramified only at $\pp$. Then if $\pp \parallel \qq$,
\begin{equation}\label{pexactly}
\Xx\left(K_0(\qq),\kbf,\sbf\right) \cap \left(\Xx\left(K_0(\qq),\kbf,\sbf\right) \otimes \omega_{\pp}\right) = \emptyset.
\end{equation}
If $\pp^2 \parallel \qq$,
\begin{multline}\label{psquared}
\Xx\left(K_0(\qq),\kbf,\sbf\right) \cap \left(\Xx\left(K_0(\qq),\kbf,\sbf\right) \otimes \omega_{\pp}\right)	\\
= \Xx\left(K_0(\qq),\kbf,\sbf\right) \setminus \left(\left(\Xx\left(K_0\left(\qq \pp^{-1}\right),\kbf,\sbf\right) \otimes \omega_{\pp}\right) \cup \left(\Xx\left(K_0\left(\qq \pp^{-2}\right),\kbf,\sbf\right) \otimes \omega_{\pp}\right)\right).
\end{multline}
If $\pp^3 \mid \qq$,
\begin{equation}\label{pcubed}
\Xx\left(K_0(\qq),\kbf,\sbf\right) \cap \left(\Xx\left(K_0(\qq),\kbf,\sbf\right) \otimes \omega_{\pp}\right) = \Xx\left(K_0(\qq),\kbf,\sbf\right).
\end{equation}
That is, when twisting automorphic representations in $\Xx(K_0(\qq),\kbf,\sbf)$ by $\omega_{\pp}$:
\begin{itemize}
	\item If $\pp \parallel \qq$, no representation has its arithmetic conductor unchanged.
	\item If $\pp^2 \parallel \qq$, the representations whose arithmetic conductors are changed are those that are twists of representations of arithmetic conductor $\qq \pp^{-1}$ or $\qq \pp^{-2}$.
	\item If $\pp^3 \mid \qq$, every representation has its arithmetic conductor unchanged.
\end{itemize}
\end{theorem}

The special case $F = \Q$ and $k_{\infty} = 0$ of \hyperref[twistthmadelic]{Theorem \ref*{twistthmadelic}} is \hyperref[twistthmclassical]{Theorem \ref*{twistthmclassical}} by the correspondence between newforms and cuspidal automorphic representations, with the archimedean weight and spectral parameter $\left(k_{\infty}, s_{\infty}\right)$ specifying the weight $k_{\infty} = 0$ and the Laplacian eigenvalue $\lambda = 1/4 - s_{\infty}^2$ of the associated newform. When $F = \Q$ and $k_{\infty} \in 2\N$, on the other hand, \hyperref[twistthmadelic]{Theorem \ref*{twistthmadelic}} reduces to a result of Atkin and Lehner \cite[Theorem 6]{Atkin1} that characterises the level of a holomorphic newform when twisted by a quadratic Dirichlet character of odd prime conductor; their proof is via classical methods, and with a little effort can be modified to prove the analogous result for Maa\ss{} newforms, namely \hyperref[twistthmclassical]{Theorem \ref*{twistthmclassical}}. On the other hand, such minor modifications do not seem to be possible to extend their work to remain valid for modular forms over an arbitrary number field $F$, as in \hyperref[twistthmadelic]{Theorem \ref*{twistthmadelic}}.

\begin{proof}[Proof of \texorpdfstring{\hyperref[twistthmadelic]{Theorem \ref*{twistthmadelic}}}{Theorem \ref{twistthmadelic}}]
Let $\pi$ be a cuspidal automorphic representation in $\Xx\left(K_0(\qq),\kbf,\sbf\right)$. As the archimedean weights and spectral parameters of $\pi$ and the conductor at unramified places are fixed under twisting by a quadratic character whose arithmetic conductor divides the arithmetic conductor of $\pi$, it suffices to determine when $c(\pi_v \otimes \beta_{v,\Quad}) = c(\pi_v)$, where $v$ is the nondyadic nonarchimedean place for which $\pp_v \cap \OO_F = \pp$ and $\pi_v$ is the local component of $\pi$ at $v$.
\begin{itemize}
\item If $\pp \parallel \qq$, the local component $\pi_v$ is a special representation $|\cdot|_v^{s_{\Quad}} \St_v$ of conductor exponent $c(\pi_v) = 1$ by \hyperref[specialpv]{Lemma \ref*{specialpv}}, and its twist by $\beta_{v,\Quad}$ has conductor exponent $c(\pi_v \otimes \beta_{v,\Quad}) = 2$ by \hyperref[specialtwist]{Lemma \ref*{specialtwist}}. This proves \eqref{pexactly}.
\item If $\pp^3 \mid \qq$, the local component $\pi_v$ is either a principal series representation or a supercuspidal representation, with the former only possible when $\pp^m \parallel \qq$ with $m \in 2\N$.
\begin{itemize}
\item If $\pi_v$ is a principal series representation, then $\pi_v \cong \beta_v |\cdot|_v^s \boxplus \beta_v^{-1} |\cdot|_v^{-s}$ for some $\beta_v$ for which $c(\beta_v) \geq 2$, so that $c(\pi_v) = 2c(\beta_v) \geq 4$. Moreover, $\beta_v \neq \beta_{v,\Quad}$ since $c(\beta_{v,\Quad}) = 1$, and so $c(\pi_v \otimes \beta_{v,\Quad}) = c(\pi_v)$ by \hyperref[principaltwist]{Lemma \ref*{principaltwist}}.
\item If $\pi_v$ is a supercuspidal representation, then \hyperref[supercuspidaltwist]{Lemma \ref*{supercuspidaltwist}} implies that $c(\pi_v \otimes \beta_{v,\Quad}) = c(\pi_v)$
\end{itemize}
This proves \eqref{pcubed}.
\item If $\pp^2 \parallel \qq$, the local component $\pi_v$ can either be a principal series representation, a special representation, or a supercuspidal representation.
\begin{itemize}
\item If $\pi_v$ is a principal series representation, then $\pi_v \cong \beta_v |\cdot|_v^s \boxplus \beta_v^{-1} |\cdot|_v^{-s}$ for some $\beta_v$ for which $c(\beta_v) = 1$, so that $c(\pi_v) = 2c(\beta_v) = 2$. If $\beta_v = \beta_{v,\Quad}$, then $\pi_v$ is the twist by $\beta_{v,\Quad}$ of the unramified principal series representation $|\cdot|_v^s \boxplus |\cdot|_v^{-s}$ by \hyperref[principaltwist]{Lemma \ref*{principaltwist}}, and otherwise the conductor of $\pi_v$ remains $\pp_v^2$ upon twisting by $\beta_{v,\Quad}$. This implies that $\Xx\left(K_0\left(\qq \pp^{-2}\right),\kbf,\sbf\right) \otimes \omega_{\pp}$ lies in $\Xx\left(K_0(\qq),\kbf,\sbf\right)$ but not in $\Xx\left(K_0(\qq),\kbf,\sbf\right) \otimes \omega_{\pp}$.
\item If $\pi_v$ is a special representation, then $\pi_v$ is the twist by $\beta_{v,\Quad}$ of a special representation $|\cdot|_v^{s_{\Quad}} \St_v$ of conductor $\pp_v$ and trivial central character by \hyperref[specialtwist]{Lemma \ref*{specialtwist}}, and so $\Xx\left(K_0\left(\qq \pp^{-1}\right),\kbf,\sbf\right) \otimes \omega_{\pp}$ lies in $\Xx\left(K_0(\qq),\kbf,\sbf\right)$ but not in $\Xx\left(K_0(\qq),\kbf,\sbf\right) \otimes \omega_{\pp}$.
\item Finally, when $\pi_v$ is a supercuspidal representation of conductor $\pp_v^2$, \hyperref[supercuspidaltwist]{Lemma \ref*{supercuspidaltwist}} shows that the twist of $\pi_v$ by $\beta_{v,\Quad}$ also has conductor $\pp_v^2$.
\end{itemize}
The identity \eqref{psquared} then follows.
\end{itemize}
\end{proof}

We can extend this result to classify cuspidal automorphic representations having trivial central character that are arithmetic conductor-invariant under twisting by a quadratic Hecke character of composite arithmetic conductor, albeit with the restriction that the quadratic Hecke character is unramified at every dyadic place.

\begin{theorem}\label{twistqthm}
Let $\qq$ and $\rr$ be integral ideals of $\OO_F$ such that $\rr \neq \OO_F$ is squarefree and every prime ideal $\pp$ dividing $\rr$ is nondyadic. Let $\omega_{\rr}$ be a quadratic Hecke character of arithmetic conductor $\rr$. Then if $\rr^2$ does not divide $\qq$, we have that
\[\left(\Xx\left(K_0(\qq),\kbf,\sbf\right) \otimes \omega_{\rr}\right) \cap \Xx\left(K_0(\qq),\kbf,\sbf\right) = \emptyset,\]
while if $\rr^2$ divides $\qq$,
\begin{multline*}
\Xx\left(K_0(\qq),\kbf,\sbf\right) \cap \left(\Xx\left(K_0(\qq),\kbf,\sbf\right) \otimes \omega_{\rr}\right)	\\
= \left\{\pi \in \Xx\left(K_0(\qq),\kbf,\sbf\right) : \text{for all $v \in S_f$ with $\pp_v \cap \OO_F \mid \rr$ and $\left(\pp_v \cap \OO_F\right)^2 \parallel \qq$,} \phantom{\left\{|\cdot|_v^{\frac{\pi i}{\log q_v}}\right\}} \hspace{-1cm}\right.	\\
\left.\pi_v \ncong \beta_{v,\Quad} |\cdot|_v^{s_{\Quad}} \boxplus \beta_{v,\Quad} |\cdot|_v^{s_{\Quad}}, \ \pi_v \ncong \beta_{v,\Quad} |\cdot|_v^{s_{\Quad}} \St_v \right.	\\
\left. \text{ with $s_{\Quad} \in \left\{0, \frac{\pi i}{\log q_v}\right\}$}\right\}.
\end{multline*}
Moreover, if $\rr' \neq \OO_F$ is a prime ideal of $\OO_F$ dividing $\rr$ and $\omega_{\rr'}$ is a quadratic Hecke character of arithmetic conductor $\rr'$, then
\[\left(\Xx\left(K_0(\qq),\kbf,\sbf\right) \otimes \omega_{\rr'}\right) \cap \Xx\left(K_0(\qq),\kbf,\sbf\right) \supset \left(\Xx\left(K_0(\qq),\kbf,\sbf\right) \otimes \omega_{\rr}\right) \cap \Xx\left(K_0(\qq),\kbf,\sbf\right).\]
\end{theorem}

\begin{proof}
This follows directly from the proof of \hyperref[twistthmadelic]{Theorem \ref*{twistthmadelic}} by classifying when $c(\pi_v \otimes \beta_{v,\Quad}) = c(\pi_v)$ for each nonarchimedean place $v$ for which $\pp_v \cap \OO_F$ divides $\rr$.
\end{proof}

\section{Counting Monomial Representations}\label{monomialsect}

Let $E$ be a quadratic extension of $F$. Via global class field theory, there is a quadratic Hecke character $\omega_{E/F}$ of $F^{\times} \backslash \A_F^{\times}$ associated to $E/F$; its arithmetic conductor is $\dd_{E/F}$, the discriminant of $E/F$. Conversely, any quadratic Hecke character $\omega_{\rr}$ of arithmetic conductor $\rr$ corresponds to a quadratic extension $E/F$ with discriminant $\dd_{E/F} = \rr$.

Recall that an automorphic representation $\pi$ of $\GL_2(\A_F)$ is said to be automorphically induced from a Hecke character $\chi$ of $E^{\times} \backslash \A_E^{\times}$ if $L(s,\chi) = L(s,\pi)$. By the work of Jacquet and Langlands \cite[\S 12]{Jacquet} (see also \cite{Shalika}), given a Hecke character $\chi$ of $E^{\times} \backslash \A_E^{\times}$, there exists a unique automorphic representation $\pi = \pi(\chi)$ of $\GL_2(\A_F)$ that is automorphically induced from $\chi$. Moreover, $\pi(\chi)$ is cuspidal if and only if $\chi$ does not factor through the global norm map $N_{\A_E / \A_F} : \A_E^{\times} \to \A_F^{\times}$, and $\pi(\chi)$ satisfies $\pi(\chi) \otimes \omega_{E/F} \cong \pi(\chi)$. The automorphically induced representation $\pi(\chi)$ is said to be monomial with respect to $\omega_{E/F}$. For $F = \Q$, this result is due to Hecke when $E$ is imaginary \cite{Hecke} and Maa\ss{} when $E$ is real \cite{Maass}.

This is the most basic form of automorphic induction. We require a converse result. Let $\pi$ be a cuspidal automorphic representation of $\GL_2(\A_F)$, and suppose that there exists some nontrivial Hecke character $\omega$ of $F^{\times} \backslash \A_F^{\times}$ such that $\pi \otimes \omega \cong \pi$. By comparing central characters, the Hecke character $\omega$ must be quadratic, and so corresponds to a quadratic extension $E$ of $F$ via class field theory.

\begin{theorem}\label{autoinductconvthm}
Let $\pi$ be a cuspidal automorphic representation of $\GL_2(\A_F)$ such that there exists a quadratic extension $E$ of $F$ for which $\pi \otimes \omega_{E/F} \cong \pi$. Then $\pi$ is monomial with respect to $\omega_{E/F}$, so that $\pi$ is automorphically induced from a Hecke character $\chi$ of $E^{\times} \backslash \A_E^{\times}$.
\end{theorem}

\begin{proof}
This is the converse to automorphic induction for cuspidal automorphic representations of $\GL_1\left(\A_E\right)$ (that is, for Hecke characters), and is a result of Labesse and Langlands \cite[Proposition 6.5]{Labesse}.
\end{proof}

\begin{lemma}\label{monomialpropertieslemma}
A monomial representation $\pi(\chi)$ with respect to $\omega_{E/F}$ has central character $\omega_{\pi(\chi)} = \omega_{E/F} \chi |_{\A_F^{\times}}$ and arithmetic conductor $\qq = N_{E/F}(\Ff) \dd_{E/F}$, where $N_{E/F} : E^{\times} \to F^{\times}$ denotes the norm map, and the integral ideal $\Ff$ of $\OO_E$ is the arithmetic conductor of $\chi$. In particular, $\dd_{E/F}$ divides $\qq$. Moreover, the arithmetic conductor of $\chi |_{\A_F^{\times}}$ is $N_{E/F}(\Ff)$.
\end{lemma}

\begin{proof}
The fact that $\qq = N_{E/F}(\Ff) \dd_{E/F}$ reduces to a local identification of conductors; this local calculation is proven in \cite[Corollary to Proposition 4 of Chapter VI]{Serre}.

To see that $\omega_{\pi(\chi)} = \omega_{E/F} \chi |_{\A_F^{\times}}$, we compare the local $L$-functions. We have that
\[\frac{1}{1 - \lambda_{\pi(\chi)_v}(\varpi_v) q_v^{-s} + \omega_{\pi(\chi_v)}(\varpi_v) q_v^{-2s}} = L(s,\pi(\chi)_v)\]
for any $v \nmid \dd_{E/F}$, where $\lambda_{\pi(\chi)_v}(\varpi_v)$ is the Hecke eigenvalue of $\pi(\chi)_v$. If $v$ splits in $E$, so that there exist two places $w_1,w_2$ lying over $v$ with $\varpi_v = \varpi_{w_1} = \varpi_{w_2}$ and $q_v = q_{w_1} = q_{w_2}$, then this local $L$-function is equal to
\[L(s,\chi_{w_1}) L(s,\chi_{w_2}) = \frac{1}{\left(1 - \chi_{w_1}(\varpi_v) q_v^{-s}\right) \left(1 - \chi_{w_2}(\varpi_v) q_v^{-s}\right)},\]
while if $v$ is inert in $E$, so that a single place $w$ lies over $v$ with $\varpi_v = \varpi_w$ and $q_v^2 = q_w$, this local $L$-function is equal to
\[L(s,\chi_w) = \frac{1}{1 - \chi_w(\varpi_v) q_v^{-2s}}.\]
Comparing coefficients of $q_v^{-2s}$ shows that $\omega_{\pi(\chi)}$ agrees with $\omega_{E/F} \chi |_{\A_F^{\times}}$ at all places not dividing $\qq$, for the local component of $\chi |_{\A_F^{\times}}$ at a place $v$ of $F$ is $\prod_{w \mid v} \chi_w |_{F_v^{\times}}$, and strong multiplicity one implies the equality of these Hecke characters at all places.

Finally, the identity for the arithmetic conductor of $\chi |_{\A_F^{\times}}$ reduces to the fact that if $\chi_w$ has conductor $\Pp_w^{c(\chi_w)}$, where $\Pp_w$ is the maximal ideal of $\OO_w$, then the conductor of $\chi_w |_{F_v^{\times}}$ is $N_{E_w / F_v}(\Pp_w^{c(\chi_w)})$, which follows from the definition of the conductor.
\end{proof}

Given a monomial cuspidal automorphic representation $\pi$ with trivial central character, let $\chi = \prod_w \chi_w$ be the Hecke character of $E^{\times} \backslash \A_E^{\times}$ from which $\pi$ is automorphically induced, where the product is over the places $w$ of $E$; we will write $S_f(E), S_{\infty}(E), S_{\R}(E), S_{\Cb}(E)$ to denote the nonarchimedean, archimedean, real, and complex places of $E$ respectively. Note that as $\chi |_{\A_F^{\times}} = \omega_{E/F}$, $N_{E/F}(\Ff)$ is equal to $\dd_{E/F}$, and hence $\qq$ is equal to $\dd_{E/F}^2$. By the work of Jacquet and Langlands, we may describe the relation between $\pi$ and $\chi$ at the archimedean places of $F$ as follows.

If $v \in S_{\R}(F)$ splits in $E$, we have that $\pi_v \cong \chi_{w_1} \boxplus \chi_{w_2}$ for the two places $w_1,w_2 \in S_{\R}(E)$ lying over $v$, with $\chi_{w_1} = \sgn_{w_1}^{m_{w_1}} |\cdot|_{w_1}^{s_{w_1}}$, $\chi_{w_2} = \sgn_{w_2}^{m_{w_2}} |\cdot|_{w_2}^{s_{w_2}}$ for some $m_{w_1}, m_{w_2} \in \Z/2\Z$ and $s_{w_1}, s_{w_2} \in i\R$; as $\pi_v$ has trivial central character, we have that $m_{w_2} = m_{w_1}$ and $s_{w_2} = -s_{w_1}$, so that the weight and spectral parameter of $\pi_v$ are $k_v = 0$ and $s_v = s_{w_1} \in i\R$.

If $v \in S_{\R}(F)$ ramifies in $E$, so that there exists one complex place $w \in S_{\Cb}(E)$ lying over $v$, we have that $\pi_v \cong \DD\left(\omega_{v,1},\omega_{v,2}\right)$ with $\omega_{v,1}, \omega_{v,2}$ characters of $F_v^{\times}$ such that $\DD\left(\omega_{v,1},\omega_{v,2}\right)$ corresponds to the character $\chi_w = e^{i m_w \arg_w} |\cdot|_w^{s_w}$ of $E_w^{\times}$, with $s_w = 0$, and $k_v = |m_w| + 1 \in 2\N$ and $s_v = m_w/2 \in 1/2 + \Z$ the weight and spectral parameter of $\pi_v$ respectively.

Finally, every complex place $v$ of $F$ splits in $E$, so that for the two places $w_1, w_2 \in S_{\Cb}(E)$ lying over $v$, we have that $\pi_v \cong \chi_{w_1} \boxplus \chi_{w_2}$ with $\chi_{w_1} = e^{i m_{w_1} \arg_{w_1}} |\cdot|_{w_1}^{s_{w_1}}$, $\chi_{w_2} = e^{i m_{w_2} \arg_{w_2}} |\cdot|_{w_2}^{s_{w_2}}$ for some $m_{w_1}, m_{w_2} \in \Z$ and $s_{w_1}, s_{w_2} \in i\R$. We have that $m_{w_2} = -m_{w_1}$ and $s_{w_2} = -s_{w_1}$, so that the weight and spectral parameter of $\pi_v$ are $k_v = 2|m_{w_1}| \in 2\N \cup \{0\}$ and $s_v = s_{w_1} \in i\R$.

Now fix a subset $S_{\R}^{\ps}(F)$ of the real places $S_{\R}(F)$ of $F$, and write $S_{\R}^{\ds}(F) \defeq S_{\R}(F) \setminus S_{\R}^{\ps}(F)$ and $S_{\infty}^{\ps}(F) \defeq S_{\R}^{\ps}(F) \cup S_{\Cb}(F)$. Let $\omega_{\rr}$ be a quadratic Hecke character of $F^{\times} \backslash \A_F^{\times}$ of arithmetic conductor $\rr$, and suppose that $\pi$ is a cuspidal automorphic representation  of $\GL_2(\A_F)$ such that $\pi_v$ is a principal series representation if $v \in S_{\infty}^{\ps}(F)$ and a discrete series representation if $v \in S_{\R}^{\ds}(F)$. Then $\pi$ can only be monomial with respect to $\omega_{\rr}$ if the quadratic extension $E$ of $F$ of discriminant $\dd_{E/F} = \rr$ associated to $\omega_{\rr}$ via class field theory is such that an archimedean place $v$ of $F$ splits in $E$ if and only if $v \in S_{\infty}^{\ps}(F)$.

We wish to count cuspidal automorphic representations of $\GL_2(\A_F)$ that are monomial with respect to $\omega_{\rr}$, have trivial central character, arithmetic conductor $\qq$, even archimedean weight vector $\kbf$, and principal series representations at a nonempty subset $S_{\infty}^{\ps}(F)$ of the archimedean places of $F$ (and therefore having discrete series representations at the remaining archimedean places $S_{\R}^{\ds}(F)$ of $F$). We let
\[\Xx\left(K_0(\qq),\kbf,S_{\infty}^{\ps}\right)_{\mon(\omega_{\rr})}\]
denote the set of all such isomorphism classes of cuspidal automorphic representations.

\begin{proposition}\label{monomialcountprop}
Let $S_{\infty}^{\ps} \subset S_{\infty}$ be nonempty and let $\kbf$ be a fixed even archimedean weight vector. Let $\qq$ and $\rr$ be integral ideals of $\OO_F$ such that $\rr^2$ divides $\qq$ and every prime ideal $\pp$ dividing $\rr$ is nondyadic, and let $\omega_{\rr}$ be a quadratic Hecke character of arithmetic conductor $\rr$. We have that
\begin{equation}\label{monomialcount}
\# \left\{\pi \in \Xx\left(K_0(\qq),\kbf,S_{\infty}^{\ps}\right)_{\mon(\omega_{\rr})} : \left|\sbf_{\ps}(\pi)\right| \leq \Tbf_{\ps}\right\} \ll_{F,\qq,\kbf} \prod_{v \in S_{\infty}^{\ps}(F)} T_v.
\end{equation}
Here we write $\left|\sbf_{\ps}(\pi)\right| \leq \Tbf_{\ps}$ to mean that $\pi$ has archimedean principal series spectral parameters $s_v\left(\pi_v\right)$ satisfying $\left|s_v\left(\pi_v\right)\right| \leq T_v$ for all $v \in S_{\infty}^{\ps}(F)$, and the implicit constant in the error term depends on $F$, $\qq$, and $\kbf$. 
\end{proposition}

\begin{proof}
Let $E$ be the quadratic extension of $F$ associated to $\omega_{\rr}$ via global class field theory, so that $\dd_{E/F} = \rr$. By \hyperref[autoinductconvthm]{Theorem \ref*{autoinductconvthm}}, it suffices to count Hecke characters $\chi$ of $E^{\times} \backslash \A_E^{\times}$ satisfying $\chi |_{\A_F^{\times}} = \omega_{\rr}$, having arithmetic conductor with relative norm $\qq \rr^{-1}$, whose associated archimedean weight vector $\kbf$ is fixed and archimedean principal series spectral parameter vector $\sbf_{\ps}$ satisfies $|\sbf_{\ps}| \leq \Tbf_{\ps}$. Note that \hyperref[monomialpropertieslemma]{Lemma \ref*{monomialpropertieslemma}} implies that no such characters exist unless $\qq = \rr^2$. Any such character $\chi$ determines an integral ideal $\Ff$ of $\OO_E$, such that $N_{E/F}(\Ff) = \qq \rr^{-1}$, and local characters $\beta_w$ of $\OO_w^{\times}$ for each $w \in S_f(E)$ for which $\Pp_w \cap \OO_E \mid \Ff$, such that $\beta_w$ descends to a primitive character of the finite abelian group $\OO_w^{\times} / \left(1 + \Pp_w^{c(\chi_w)}\right)$ when restricted to $\OO_w^{\times}$; we observe that there are only finitely many such ideals and, in turn, only finitely many such characters.

To count all such $\chi$, we first fix an ideal $\Ff$ satisfying $N_{E/F}(\Ff) = \qq \rr^{-1}$ and a set of characters $\beta_w$ for each $w \in S_f(E)$ for which $\Pp_w \cap \OO_E \mid \Ff$; we then bound the number of such Hecke characters $\chi = \prod_w \chi_w$ satisfying the conditions of the previous paragraph that additionally have arithmetic conductor $\Ff$ and whose local components are of the form $\chi_w = \beta_w |\cdot|_w^{s_w}$ for each $w \in S_f(E)$ for which $\Pp_w \cap \OO_E \mid \Ff$, where $s_w \in i\R$. We then sum this bound over the finitely many possible characters $\beta_w$ and the finitely many possible integral ideals $\Ff$.

At the archimedean places, if $v \in S_{\R}^{\ps}(F)$ with $w_1, w_2 \mid v$, then $k_v = |m_{w_1} - m_{w_2}| = 0$ and $s_v = (s_{w_1} - s_{w_2})/2 = s_{w_1}$, while if $v \in S_{\R}^{\ds}(F)$ with $w \mid v$, then $k_v = |m_w| + 1$, $s_v = m_w/2$, and $s_w = 0$, and finally $k_v = |m_{w_1} - m_{w_2}| = 2|m_{w_1}| \in 2 \N \cup \{0\}$ and $s_v = (s_{w_1} - s_{w_2})/2 = s_{w_1} \in i\R$ if $v \in S_{\Cb}(F)$ with $w_1, w_2 \mid v$.

The fact that $\chi$ is trivial on $E^{\times}$ implies that for each unit $\epsilon$ of $\OO_E$,
\[1 = \prod_{w \in S_{\R}(E)} \sgn_w(\epsilon_w)^{m_w} |\epsilon_w|_w^{s_w} \prod_{w \in S_{\Cb}(E)} e^{i m_w \arg_w(\epsilon_w)} |\epsilon_w|_w^{s_w} \prod_{\substack{w \in S_f(E) \\ \Pp_w \cap \OO_E \mid \Ff}} \beta_w(\epsilon_w),\]
where we write $\epsilon_w$ to denote the image of $\epsilon$ in the local field $E_w$, and for $w \in S_f(E)$, $\Pp_w$ is the maximal ideal of $\OO_w$. It follows that
\begin{multline}\label{fundunitcondition}
\sum_{\substack{v \in S_{\infty}^{\ps}(F) \\ w_1, w_2 \mid v}} s_v \log \frac{|\epsilon_{w_1}|_{w_1}}{|\epsilon_{w_2}|_{w_2}} + \sum_{w \in S_{\R}(E)} m_w \log \sgn_w(\epsilon_w)	\\
+ i \sum_{w \in S_{\Cb}(E)} m_w \arg_w(\epsilon_w) + \sum_{\substack{w \in S_f(E) \\ \Pp_w \cap \OO_E \mid \Ff}} \log \beta_w(\epsilon_w) \in 2\pi i \Z
\end{multline}
for a fixed branch of the logarithm that ensures this is well-defined; by assumption, all but the first term is fixed.

We will show in \hyperref[latticelemma]{Lemma \ref*{latticelemma}} that under the assumption that every prime ideal $\pp$ dividing $\rr$ is nondyadic,
\[\left\{\left(\log \frac{|\epsilon_{w_1}|_{w_1}}{|\epsilon_{w_2}|_{w_2}}\right)_{w_1, w_2 \mid v \in S_{\infty}^{\ps}(F)} \in \R^{\# S_{\infty}^{\ps}(F)} : \epsilon \in \OO_E^{\times}\right\}\]
is a lattice in $\R^{\# S_{\infty}^{\ps}(F)}$, so that
\[\left\{\sbf_{\ps} \in (i\R)^{\# S_{\infty}^{\ps}(F)} : \sum_{\substack{v \in S_{\infty}^{\ps}(F) \\ w_1, w_2 \mid v}} s_v \log \frac{|\epsilon_{w_1}|_{w_1}}{|\epsilon_{w_2}|_{w_2}} \in 2\pi i \Z \text{ for all $\epsilon \in \OO_E^{\times}$}\right\}\]
is a dual lattice in $(i\R)^{\# S_{\infty}^{\ps}(F)}$, and
\[\left\{\sbf_{\ps} \in (i\R)^{\# S_{\infty}^{\ps}(F)} : \text{\eqref{fundunitcondition} holds for all $\epsilon \in \OO_E^{\times}$}\right\}\]
is the translate of this dual lattice by a fixed vector. The result then follows by standard bounds for the number of points in a lattice lying in a rectangle, and summing over the finitely many possible combinations of local characters $\beta_w$ for each $w \in S_f(E)$ for which $\Pp_w \cap \OO_E \mid \Ff$ and over the finitely many possible integral ideals $\Ff$ of $\OO_E$ for which $N_{E/F}(\Ff) = \qq \rr^{-1}$.
\end{proof}

\begin{lemma}\label{latticelemma}
Let $E/F$ be a quadratic extension that is ramified at a nondyadic nonarchimedean place and is such that the set $S_{\infty}^{\ps}(F)$ of archimedean places of $F$ that split in $E$ is nonempty, and fix an ordering $w_1,w_2$ of the places in $E$ lying over a place $v \in S_{\infty}^{\ps}(F)$. Consider the map $\Phi : \OO_E^{\times} \to \R^{\# S_{\infty}^{\ps}(F)}$ given by
\[\Phi(\epsilon) \defeq \left(\log \frac{|\epsilon_{w_1}|_{w_1}}{|\epsilon_{w_2}|_{w_2}}\right)_{w_1, w_2 \mid v \in S_{\infty}^{\ps}(F)}.\]
Then $\Phi$ is a homomorphism with kernel $\mu(\OO_E) \OO_F^{\times}$, with $\mu(\OO_E)$ the set of roots of unity lying in $E$, and $\OO_F^{\times}$ viewed as a subgroup of $\OO_E^{\times}$. In particular, the image of $\Phi$ is a lattice in $\R^{\# S_{\infty}^{\ps}(F)}$.
\end{lemma}

\begin{proof}
It is clear that $\Phi$ is a homomorphism and that $\mu(\OO_E) \OO_F^{\times}$ lies in the kernel of $\Phi$. To prove the reverse containment, we let $\sigma$ be the nontrivial automorphism of $E/F$ in $\Gal(E/F)$. This acts transitively on the set of extensions $E_w$ lying over $F_v$ for any place $v$ of $F$, so that for any $x \in E$, if a place $v \in S_{\infty}(F)$ splits in $E$ with $w_1,w_2$ lying over $v$, then $\sigma(x_{w_2}) = x_{w_1}$, while if $v$ ramifies in $E$ with $w$ lying over $v$, then $\sigma(x_w) = x_w$. It follows that $\epsilon \in \ker \Phi$ if and only if
\[\mu \defeq \epsilon \sigma(\epsilon)^{-1}\]
satisfies $|\mu_w|_w = 1$ for every archimedean place $w$ of $E$. As $\epsilon \in \OO_E^{\times}$ implies that $\sigma(\epsilon) \in \OO_E^{\times}$, so that $\mu \in \OO_E^{\times}$, it follows that this is only possible should $\mu$ be a root of unity in $E$.

Let $N = mn$ be the order of the root of unity $\mu$, with $m$ a positive odd integer and $n$ a power of $2$. By B\'{e}zout's identity, there exist integers $a,b$ such that $am + bn = 1$. We may then factor $\mu$ as $\mu = \mu^{am} \mu^{bn}$, with $\mu^{am}, \mu^{bn}$ roots of unity in $\mu\left(\OO_E\right)$ such that $\mu^{am}$ is of order $n$ and $\mu^{bn}$ is of order $m = 2k + 1$.

We define the unit $\epsilon' \defeq \epsilon \mu^{-(k + 1)bn}$ of $\OO_E$; it satisfies
\[\sigma\left(\epsilon'\right) = \sigma(\epsilon) \sigma(\mu)^{-(k + 1)bn} = \epsilon \mu^{-1} \mu^{(k + 1)bn} = \epsilon' \mu^{-am}\]
by the fact that $\sigma(\mu) = \mu^{-1}$ and that $\mu^{2(k + 1)bn} = \mu^{bn}$. It follows that $\sigma\left({\epsilon'}^{n}\right) = {\epsilon'}^n$, and hence ${\epsilon'}^n$ lies in $F$. We claim that $\epsilon' \in F$, in which case it must be a unit of $\OO_F$, and we deduce that $\epsilon \in \mu\left(\OO_F\right) \OO_F^{\times}$, as desired.

Indeed, suppose in order to obtain a contradiction that the least power $r$ of $2$ for which ${\epsilon'}^{r/2} \notin F$ is greater than $1$, and consider the extension $F\left({\epsilon'}^{r/2}\right)$ of $F$; as ${\epsilon'}^r \in F$, this is a quadratic extension, and as ${\epsilon'}^{r/2} \in E$, this extension must be equal to $E$. On the other hand, this extension is generated by the square root of a unit ${\epsilon'}^r$ of $\OO_F$, and hence is unramified at every nondyadic nonarchimedean place, contradicting our assumption on $r$; consequently, $\epsilon' \in F$. To see this, note that at a nondyadic nonarchimedean place $v$, if ${\epsilon'}^r$ is a square in $F_v$ then $F\left({\epsilon'}^{r/2}\right)$ splits at $v$ and so is unramified there. Otherwise $F_v\left({\epsilon'}^{r/2}\right)$ is the splitting field over $F_v$ of $f(x) = x^2 - {\epsilon'}^r$. As ${\epsilon'}^r$ is a unit in $F_v$, this is coprime to $f'(x) = 2x$ (noting that $(2,q_v) = 1$), and so the discriminant of $F_v\left({\epsilon'}^{r/2}\right) / F_v$ is contained in the ideal generated by the discriminant of $f(x)$, which is a unit.

Finally, the fact that the image of $\Phi$ is a lattice follows from Dirichlet's unit theorem: $\OO_E^{\times}$ and $\ker \Phi = \mu(\OO_E) \OO_F^{\times}$ are finitely generated abelian groups of rank $\# S_{\infty}(E) - 1$ and $\# S_{\infty}(F) - 1$ respectively, which implies that $\OO_E^{\times} / \ker \Phi$ is a finitely generated abelian group of rank
\[\# S_{\infty}(E) - 1 - \# S_{\infty}(F) + 1 = \# S_{\infty}^{\ps}(F).\qedhere\]
\end{proof}

\section{Weyl Laws for Automorphic Representations of \texorpdfstring{$\GL_2(\A_F)$}{GL\9040\202(AF)}}\label{Weyllawsect}

Let $v$ be a nonarchimedean place of an algebraic number field $F$, and let $\pi_v, \pi_v'$ be generic irreducible admissible unitarisable representations of $\GL_2(F_v)$. We say that $\pi_v$ and $\pi_v'$ are similar if they have the same central character and are either
\begin{itemize}
\item both supercuspidal representations that are isomorphic up to twisting by the unramified character $|\cdot|_v^{\frac{\pi i}{\log q_v}}$,
\item both special representations that are isomorphic up to twisting by the unramified character $|\cdot|_v^{\frac{\pi i}{\log q_v}}$, or
\item both principal series representations
\[\pi_v \cong \beta_{v,1} |\cdot|_v^{s_1} \boxplus \beta_{v,2} |\cdot|_v^{s_2}, \qquad \pi_v' \cong \beta_{v,1}' |\cdot|_v^{s_1'} \boxplus \beta_{v,2}' |\cdot|_v^{s_2'}\]
with $s_1 + s_2 - s_1' - s_2' \in \frac{2\pi i}{\log q_v} \Z$ and either $\beta_{v,1}' = \beta_{v,1}$ and $\beta_{v,2}' = \beta_{v,2}$, or $\beta_{v,1}' = \beta_{v,2}$ and $\beta_{v,2}' = \beta_{v,1}$.
\end{itemize}
Likewise, for a real place $v$ and generic irreducible admissible unitarisable $\left(\gl_2\left(F_v\right)_{\Cb}, K_v\right)$-modules $\pi_v, \pi_v'$, we say that $\pi_v$ and $\pi_v'$ are similar if they are either
\begin{itemize}
\item both isomorphic discrete series representations, or
\item both principal series representations
\[\pi_v \cong \sgn_v^{m_1} |\cdot|_v^{s_1} \boxplus \sgn_v^{m_2} |\cdot|_v^{s_2}, \qquad \pi_v' \cong \sgn_v^{m_1'} |\cdot|_v^{s_1'} \boxplus \sgn_v^{m_2'} |\cdot|_v^{s_2'}\]
with $s_1 + s_2 = s_1' + s_2'$ and equal weights $k_v = |m_1 - m_2|$, $k_v' = |m_1' - m_2'|$.
\end{itemize}
Finally, two generic irreducible admissible unitarisable $\left(\gl_2\left(F_v\right)_{\Cb}, K_v\right)$-modules $\pi_v, \pi_v'$, with $v$ a complex place of $F$, are said to be similar if they are both principal series representations
\[\pi_v \cong e^{i m_1 \arg_v} |\cdot|_v^{s_1} \boxplus e^{i m_2 \arg_v} |\cdot|_v^{s_2}, \qquad \pi_v' \cong e^{i m_1' \arg_v} |\cdot|_v^{s_1'} \boxplus e^{i m_2' \arg_v} |\cdot|_v^{s_2'}\]
with $s_1 + s_2 = s_1' + s_2'$ and equal weights $k_v = |m_1 - m_2|$, $k_v' = |m_1' - m_2'|$.

We denote by $\Xx_v$ a local similarity class of generic irreducible admissible unitarisable representations of $\GL_2(F_v)$ or $\left(\gl_2\left(F_v\right)_{\Cb}, K_v\right)$-modules. A global similarity class $\Xx \defeq \left(\Xx_v\right)$ of cuspidal automorphic representations of $\GL_2(\A_F)$ is then defined to be the set of all isomorphism classes of cuspidal automorphic representations $\pi$ of $\GL_2(\A_F)$ for which $\pi_v \in \Xx_v$ for each local component $\pi_v$ of $\pi$. Note that two elements $\pi,\pi'$ of a global similarity class $\Xx$ may well have different principal series spectral parameters $s_v$ at each place $v$ for which $\Xx_v$ is a similarity class of principal series representations, but that they share the same central character $\omega_{\pi}$, arithmetic conductor $\qq$, and archimedean weight vector $\kbf$.

In \cite{Palm}, Palm chooses a test function for the ad\`{e}lic Arthur--Selberg trace formula in order to produce a Weyl law for global similarity classes $\Xx$ of cuspidal automorphic representations; that is, he gives an asymptotic estimate for the number of elements of $\Xx$ for which $|\sbf_{\ps}| \leq \Tbf_{\ps}$.

\begin{proposition}[Weyl Law for $\GL_2(\A_F)$ {\cite[Theorem 3.2.1]{Palm}}]\label{Weyllaw}
Let $\Xx$ be a global similarity class of cuspidal automorphic representations of $\GL_2(\A_F)$, and let $S_{\infty}^{\ps}$ be a nonempty subset of the archimedean places of $F$ such that $\Xx_v$ is a similarity class of principal series representations for each $v \in S_{\infty}^{\ps}$. Suppose further that at each complex place $v$, the archimedean weight of $\Xx_v$ is $k_v = 0$. Then we have the Weyl law
\begin{multline*}
\#\left\{\pi \in \Xx : \left|\sbf_{\ps}(\pi)\right| \leq \Tbf_{\ps}\right\}	\\
= C_{\Xx} \prod_{v \in S_{\R}^{\ps}} T_v^2 \prod_{v \in S_{\Cb}} T_v^3 + O_F\left(C_{\Xx} \sum_{v' \in S_{\infty}^{\ps}} \frac{1}{T_{v'}} \prod_{v \in S_{\R}^{\ps}} T_v^2 \prod_{v \in S_{\Cb}} T_v^3 + \sum_{v' \in S_{\R}^{\ps}} \log (T_{v'}) \prod_{v \in S_{\infty}^{\ps}} T_v\right).
\end{multline*}
Here the implicit constant in the error term depends only on the algebraic number field $F$, and the constant $C_{\Xx}$ is defined by
\[C_{\Xx} \defeq \frac{\vol\left(Z\left(\A_F\right) \GL_2(F) \backslash \GL_2\left(\A_F\right)\right)}{(4\pi)^{\# S_{\R}} (8\pi^2)^{\# S_{\Cb}}} \prod_{v \notin S_{\infty}^{\ps}} C_{\Xx_v}, \]
with each local constant $C_{\Xx_v}$ defined such that if $\Xx_v$ is not a similarity class of principal series representations, then for any $\pi_v \in \Xx_v$,
\[C_{\Xx_v} \defeq \begin{dcases*}
q_v - 1 & if $\pi_v$ is a special representation,	\\
\dim \rho_{\pi_v} & if $\pi_v$ is a type I supercuspidal representation,	\\
\frac{q_v + 1}{2} \dim \rho_{\pi_v} & if $\pi_v$ is a type II supercuspidal representation,	\\
k_v - 1 & if $\pi_v$ is a discrete series representation of weight $k_v$,
\end{dcases*}\]
while if $\Xx_v$ is a similarity class of principal series representations, then for any $\pi_v \cong \omega_{v,1} \boxplus \omega_{v,2} \in \Xx_v$,
\[C_{\Xx_v} \defeq \begin{dcases*}
q_v^{c\left(\omega_{v,1} \omega_{v,2}^{-1}\right) - 1} \left(q_v + 1\right) & if $c\left(\omega_{v,1} \omega_{v,2}^{-1}\right) \geq 1$,	\\
1 & otherwise.
\end{dcases*}\]
\end{proposition}

We have corrected several typographical errors in \cite[Theorem 3.2.1]{Palm}. The final error term should be $\sum_{v' \in S_{\R}^{\ps}} \log T_{v'} \prod_{v \in S_{\infty}^{\ps}} T_v$, not $\sum_{v' \in S_{\R}^{\ps}} \log T_{v'} \prod_{v \in S_{\infty}} T_v$. By \cite[Proposition 7.7.1, Corollary 7.7.2, Proposition 8.7.1]{Palm}, the denominator in the definition of $C_{\Xx}$ should be $(4\pi)^{\# S_{\R}} (8\pi^2)^{\# S_{\Cb}}$, not $(4\pi)^{\# S_{\infty}}$. The definition of $C_{\Xx_v}$ for $\Xx_v$ a similarity class of discrete series representations of weight $k_v$ ought to be $k_v - 1$ as in \cite[Corollary 7.7.2]{Palm}, not $\frac{k_v - 1}{2}$. Finally, for $\Xx_v$ a similarity class of principal series representations with $c\left(\omega_{v,1} \omega_{v,2}^{-1}\right) \geq 1$, the definition of $C_{\Xx_v}$ is incorrect in both \cite[Theorem 3.2.1]{Palm} and its derivation in \cite[Proposition 9.7.1]{Palm}; there is no need to normalise $C_{\Xx_v}$ by dividing by $q_v^{\lfloor c\left(\omega_{v,1} \omega_{v,2}^{-1}\right) / 2 \rfloor} + 1$.

We let
\[\Xx\left(K_0(\qq),\kbf,S_{\infty}^{\ps}\right)\]
denote the set of isomorphism classes of cuspidal automorphic representations of arithmetic conductor $\qq$, trivial central character, archimedean weight vector $\kbf$, having principal series representations at a fixed nonempty subset $S_{\infty}^{\ps}$ of the archimedean places of $F$ containing all the complex places, and having discrete series representations at the remaining real archimedean places $S_{\R}^{\ds} \defeq S_{\infty} \setminus S_{\infty}^{\ps}$; this is a union of similarity classes. For a quadratic Hecke character $\omega_{\rr}$ of arithmetic conductor $\rr$, we define
\begin{multline*}
\Xx\left(K_0(\qq),\kbf,S_{\infty}^{\ps}\right)_{\nonmon(\omega_{\rr})}	\\
\defeq \left\{\pi \in \Xx\left(K_0(\qq),\kbf,S_{\infty}^{\ps}\right) : \pi \otimes \omega_{\rr} \ncong \pi,	\ \pi \otimes \omega_{\rr} \in \Xx\left(K_0(\qq),\kbf,S_{\infty}^{\ps}\right) \right\};
\end{multline*}
these are the isomorphism classes of $\omega_{\rr}$-nonmonomial cuspidal automorphic representations in $\Xx\left(K_0(\qq),\kbf,S_{\infty}^{\ps}\right)$ that are arithmetic conductor-invariant under twisting by $\omega_{\rr}$. Using the Weyl law, we are able to count such cuspidal automorphic representations, and hence prove instances of spectral multiplicity for modular forms over arbitrary number fields.

\begin{theorem}\label{Fratiothm}
Let $S_{\infty}^{\ps} \subset S_{\infty}$ be nonempty and $\kbf$ be a fixed even archimedean weight vector with $k_v = 0$ for all $v \in S_{\infty}^{\ps}$. Let $\qq$ and $\rr$ be integral ideals of $\OO_F$ such that $\rr \neq \OO_F$ is squarefree and every prime ideal $\pp$ dividing $\rr$ is nondyadic. Let $\omega_{\rr}$ be a quadratic Hecke character of arithmetic conductor $\rr$. Then if $\rr^2$ does not divide $\qq$,
\begin{equation}\label{Frationdivide}
\Xx\left(K_0(\qq),\kbf,S_{\infty}^{\ps}\right)_{\nonmon(\omega_{\rr})} = \emptyset,
\end{equation}
while if $\rr^2$ divides $\qq$, we have the Weyl law
\begin{multline}\label{Fratiodivide}
\frac{\# \left\{\pi \in \Xx\left(K_0(\qq),\kbf,S_{\infty}^{\ps}\right)_{\nonmon(\omega_{\rr})} : \left|\sbf_{\ps}(\pi)\right| \leq \Tbf_{\ps}\right\}}{\# \left\{\pi \in \Xx\left(K_0(\qq),\kbf,S_{\infty}^{\ps}\right) : \left|\sbf_{\ps}(\pi)\right| \leq \Tbf_{\ps}\right\}}	\\
= \prod_{\substack{v \in S_f \\ \pp_v \cap \OO_F \mid \rr \\ \left(\pp_v \cap \OO_F\right)^2 \parallel \qq}} \left(1 - \frac{q_v}{q_v^2 - q_v - 1}\right) + o_{F,\qq,\kbf}(1)
\end{multline}
as $T_v$ tends to infinity for every $v \in S_{\infty}^{\ps}$. Here the error term depends on $F$,$\qq$, and $\kbf$, and we have the convention that the empty product is equal to $1$.

Moreover, the same holds if we replace $\Xx\left(K_0(\qq),\kbf,S_{\infty}^{\ps}\right)_{\nonmon(\omega_{\rr})}$ by
\begin{equation}\label{intersection}
\bigcap_{\substack{\rr' \mid \rr \\ \rr' \neq \OO_F}} \bigcap_{\substack{\omega_{\rr'} \\ \cond\left(\omega_{\rr'}\right) = \rr'}} \Xx\left(K_0(\qq),\kbf,S_{\infty}^{\ps}\right)_{\nonmon(\omega_{\rr'})},
\end{equation}
where $\cond(\omega)$ denotes the arithmetic conductor of $\omega$, and the intersection is over all integral ideals $\rr' \neq \OO_F$ of $\OO_F$ dividing $\rr$ and all quadratic Hecke characters $\omega_{\rr'}$ of arithmetic conductor $\rr'$.
\end{theorem}

The special case $F = \Q$ and $k_{\infty} = 0$ of \hyperref[Fratiothm]{Theorem \ref*{Fratiothm}} is \hyperref[Qratiothm]{Theorem \ref*{Qratiothm}} by the bijective correspondence between newforms and cuspidal automorphic representations. Note in this case that there is a unique primitive quadratic Dirichlet character modulo $r'$ for every positive odd integer $r'$ greater than $1$.

\begin{proof}
It is clear from \hyperref[twistqthm]{Theorem \ref*{twistqthm}} that \eqref{Frationdivide} holds when $\rr^2$ does not divide $\qq$. If $\rr^2$ divides $\qq$, then the $\omega_{\rr}$-monomial representations have density zero in $\Xx\left(K_0(\qq),\kbf,S_{\infty}^{\ps}\right)$ by \hyperref[monomialcountprop]{Propositions \ref*{monomialcountprop}} and \ref{Weyllaw}. Moreover, \hyperref[Weyllaw]{Proposition \ref*{Weyllaw}} implies that
\begin{multline*}
\# \left\{\pi \in \Xx\left(K_0(\qq),\kbf,S_{\infty}^{\ps}\right) : \left|\sbf_{\ps}(\pi)\right| \leq \Tbf_{\ps}\right\}	\\
\sim \frac{\vol\left(Z\left(\A_F\right) \GL_2(F) \backslash \GL_2\left(\A_F\right)\right)}{(4\pi)^{\# S_{\R}} (8\pi^2)^{\# S_{\Cb}}} \prod_{\substack{v \in S_f \\ (\pp_v \cap \OO_F)^r \parallel \qq}} \sum_{\substack{\Xx_v : \exists \pi_v \in \Xx_v \text{ s.t} \\ c(\pi_v) = r, \ \omega_{\pi_v} = 1}} C_{\Xx_v}	\\
\times \prod_{v \in S_{\R}^{\ds}} (k_v - 1) \prod_{v \in S_{\R}^{\ps}} T_v^2 \prod_{v \in S_{\Cb}^{\ps}} T_v^3
\end{multline*}
as $T_v$ tends to infinity for every $v \in S_{\infty}^{\ps}$. Similarly, \hyperref[Weyllaw]{Proposition \ref*{Weyllaw}} coupled with \hyperref[twistqthm]{Theorem \ref*{twistqthm}} implies that
\[\# \left\{\pi \in \Xx\left(K_0(\qq),\kbf,S_{\infty}^{\ps}\right)_{\nonmon(\omega_{\rr})} : \left|\sbf_{\ps}(\pi)\right| \leq \Tbf_{\ps}\right\}\]
is asymptotic to
\begin{multline*}
\frac{\vol\left(Z\left(\A_F\right) \GL_2(F) \backslash \GL_2\left(\A_F\right)\right)}{(4\pi)^{\# S_{\R}} (8\pi^2)^{\# S_{\Cb}}} \prod_{\substack{v \in S_f \\ (\pp_v \cap \OO_F)^r \parallel \qq \\ \pp_v \cap \OO_F \nmid \rr \text{ or } r \neq 2}} \sum_{\substack{\Xx_v : \exists \pi_v \in \Xx_v \text{s.t} \\ c(\pi_v) = r, \ \omega_{\pi_v} = 1}} C_{\Xx_v}	\\
\times \prod_{\substack{v \in S_f \\ \pp_v \cap \OO_F \mid \rr \\ (\pp_v \cap \OO_F)^2 \parallel \qq}} \left(\sum_{\substack{\Xx_v : \exists \pi_v \in \Xx_v \text{ s.t.} \\ c\left(\pi_v\right) = 2, \ \omega_{\pi_v} = 1}} C_{\Xx_v} - \sum_{\substack{\Xx_v : \exists \pi_v \in \Xx_v \text{ s.t.} \\ \pi_v \cong \beta_{v,\Quad} \boxplus \beta_{v,\Quad}}} C_{\Xx_v} - \sum_{\substack{\Xx_v : \exists \pi_v \in \Xx_v \text{ s.t.} \\ \pi_v \cong \beta_{v,\Quad} \St_v}} C_{\Xx_v}\right)	\\
\times \prod_{v \in S_{\R}^{\ds}} (k_v - 1) \prod_{v \in S_{\R}^{\ps}} T_v^2 \prod_{v \in S_{\Cb}^{\ps}} T_v^3.
\end{multline*}
So to show \eqref{Fratiodivide}, we must merely determine the ratio of the quantities
\begin{equation}\label{nonmonconstant}
\sum_{\substack{\Xx_v : \exists \pi_v \in \Xx_v \text{ s.t.} \\ c\left(\pi_v\right) = 2, \ \omega_{\pi_v} = 1}} C_{\Xx_v} - \sum_{\substack{\Xx_v : \exists \pi_v \in \Xx_v \text{ s.t.} \\ \pi_v \cong \beta_{v,\Quad} \boxplus \beta_{v,\Quad}}} C_{\Xx_v} - \sum_{\substack{\Xx_v : \exists \pi_v \in \Xx_v \text{ s.t.} \\ \pi_v \cong \beta_{v,\Quad} \St_v}} C_{\Xx_v},
\end{equation}
and
\begin{equation}\label{constant}
\sum_{\substack{\Xx_v : \exists \pi_v \in \Xx_v \text{ s.t.} \\ c\left(\pi_v\right) = 2, \ \omega_{\pi_v} = 1}} C_{\Xx_v}
\end{equation}
for each place $v \in S_f$ with $\pp_v \cap \OO_F \mid \rr$ and $\left(\pp_v \cap \OO_F\right)^2 \parallel \qq$. The product of these ratios over all such places will then yield \eqref{Fratiodivide}. Finally, the fact that we can replace $\Xx\left(K_0(\qq),\kbf,S_{\infty}^{\ps}\right)_{\omega_{\rr}}$ by \eqref{intersection} again follows from \hyperref[twistqthm]{Theorem \ref*{twistqthm}} and the fact that the monomial representations have density zero.

From the definition of $C_{\Xx_v}$, the second term in \eqref{nonmonconstant} is equal to $1$, while the third is equal to $q_v - 1$. So it remains to determine \eqref{constant}, which we do by dividing into cases.
\begin{itemize}
\item There is only one local similarity class $\Xx_v$ of principal series representations with $\pi_v \cong \omega_{v,1} \boxplus \omega_{v,2} \in \Xx_v$ satisfying $c\left(\omega_{v,1} \omega_{v,2}^{-1}\right) = 0$, namely that containing $\pi_v \cong \beta_{v,\Quad} \boxplus \beta_{v,\Quad}$; this has $C_{\Xx_v} = 1$.
\item The remaining local similarity classes of principal series representations that can occur in the summation index of \eqref{constant} are those containing $pi_v \cong \omega_v \boxplus \omega_v^{-1}$ with $c\left(\omega_v\right) = c\left(\omega_v^2\right) = 1$, so that $c(\pi_v) = 2$ and $\pi_v \ncong \beta_{v,\Quad} \boxplus \beta_{v,\Quad}$. These local similarity classes $\Xx_v$ are parametrised by the squares of characters $\beta_v$ of $\OO_v^{\times} / \left(1 + \pp_v\right)$ with $\beta_v^2 \neq 1$; there are
\[\frac{\# \OO_v^{\times} / \left(1 + \pp_v\right)}{2} - 1 = \frac{q_v - 3}{2}\]
distinct squares of these characters, with $C_{\Xx_v} = q_v + 1$ in each case.
\item There is only one local similarity class $\Xx_v$ of special representations of conductor $\pp_v^2$ and trivial central character, namely that containing $\pi_v \cong \beta_{v,\Quad} \St_v$; this has $C_{\Xx_v} = q_v - 1$.
\item Finally, Knightly and Ragsdale \cite{Knightly} show that every supercuspidal representation $\pi_v$ of conductor $\pp_v^2$ is of type I with $\dim \rho_{\pi_v} = q_v - 1$, and the number of isomorphism classes of such representations having trivial central character is $\frac{q_v - 1}{2}$. So for each such local similarity class $\Xx_v$, we have that $C_{\Xx_v} = q_v - 1$.
\end{itemize}
Combining these calculations, we find that
\[\sum_{\substack{\Xx_v : \exists \pi_v \in \Xx_v \text{ s.t.} \\ c\left(\pi_v\right) = 2, \ \omega_{\pi_v} = 1}} C_{\Xx_v} = q_v - 1 + 1 + \frac{q_v - 3}{2} \left(q_v + 1\right) + \frac{q_v - 1}{2} \left(q_v - 1\right) = q_v^2 - q_v - 1,\]
and that
\[\sum_{\substack{\Xx_v : \exists \pi_v \in \Xx_v \text{ s.t.} \\ c\left(\pi_v\right) = 2, \ \omega_{\pi_v} = 1}} C_{\Xx_v} - \sum_{\substack{\Xx_v : \exists \pi_v \in \Xx_v \text{ s.t.} \\ \pi_v \cong \beta_{v,\Quad} \boxplus \beta_{v,\Quad}}} C_{\Xx_v} - \sum_{\substack{\Xx_v : \exists \pi_v \in \Xx_v \text{ s.t.} \\ \pi_v \cong \beta_{v,\Quad} \St_v}} C_{\Xx_v} = q_v^2 - 2q_v - 1.
\qedhere\]
\end{proof}

\section{Consequences and Conjectures}

\subsection{Spectral Multiplicity}

While our results give lower bounds for spectral multiplicity, our methods say very little about corresponding upper bounds. In light of the numerical evidence of the simplicity of the cuspidal spectrum of the Laplacian on $\SL_2(\Z) \backslash \Hb$, however, we believe that spectral multiplicity should only occur due to quadratic twisting, with the possible exception of the eigenvalue $1/4$.

\begin{conjecture}[Spectral Multiplicity Conjecture for $\Gamma_0(q) \backslash \Hb$]
Let $q$ be a positive odd integer, and let $s(q)$ denote the number of distinct odd primes $p$ for which $p^2$ divides $q$. Then the multiplicity of an eigenvalue $\lambda > 1/4$ in the new part of the cuspidal spectrum of the Laplacian on $\Gamma_0(q) \backslash \Hb$ is bounded from above by $2^{s(q)}$, with this bound attained for a positive proportion of eigenvalues.
\end{conjecture}

When $q$ is even, the situation is more complicated, though we point out that the cuspidal spectrum of the Laplacian on $\Gamma_0(q) \backslash \Hb$ with $q$ even has been studied via the Selberg trace formula by Golovchanski\u{\i} and Smotrov \cite{Golovchanskii}. Throughout this paper we have stipulated that all quadratic Dirichlet characters have odd conductor. Classifying the level of twists of newforms by quadratic characters of even conductor has been considered by Atkin and Lehner \cite[Theorem 7]{Atkin1} for holomorphic newforms, with the chief challenge being that there is no quadratic character modulo $2$, one quadratic character modulo $4$, and two primitive quadratic characters modulo $8$. It is likely that, at least for the case $F = \Q$, the methods of this paper could be used to extend results on spectral multiplicity to even values of $q$. More generally, one can undoubtedly classify the generic irreducible admissible representations of $\GL_2(F_v)$ with $v$ dyadic whose conductors are invariant under twisting by a given ramified quadratic character of $F_v^{\times}$ via the same methods as in \hyperref[backgroundsect]{Section \ref*{backgroundsect}}, though additional work would be required due to the fact that such ramified quadratic characters can have conductor exponent $2$, not just $1$. On the other hand, more work would be required in bounding the contribution of monomial representations, since the proof of \hyperref[latticelemma]{Lemma \ref*{latticelemma}} seems to rely on the fact that the twisting character $\omega_{\rr}$ is ramified at a nondyadic nonarchimedean place; note that for $F = \Q$, however, there are simpler ways to show that monomial representations contribute a negligible amount.

\subsection{Multiplicity of Hecke Eigenvalues and Distinguishing Newforms}

The results in this paper can be used not only to prove multiplicity of Laplacian eigenvalues but also to prove the nonarchimedean analogue, namely multiplicity of Hecke eigenvalues. More precisely, we have the following.

\begin{proposition}
Let $q$ be a positive integer, and let $p$ be a fixed prime not dividing $q$. Let $m_p(q)$ denote the number of odd squarefree divisors $r' > 1$ of $q$ such that ${r'}^2$ divides $q$ and that $\chi_{r'}(p) = 1$, where $\chi_{r'}$ denotes the unique primitive quadratic character modulo $r'$. Then as $T$ tends to infinity, a positive proportion of the Hecke eigenvalues $\lambda_f(p)$ of Maa\ss{} newforms $f \in \BB_0^{\ast}\left(\Gamma_0(q)\right)$ with Laplacian eigenvalue $\lambda_f \leq T^2$ have multiplicity at least $m_p(q) + 1$.
\end{proposition}

Indeed, \hyperref[Qratiothm]{Theorem \ref*{Qratiothm}} shows that if $q$ is not squarefree, then for each odd squarefree integer $r$ such that $r^2$ divides $q$ there exists a positive proportion of Maa\ss{} newforms $f,g$ of weight $0$, level $q$, and principal character that are not equal yet satisfy $f \otimes \chi_r = g$, where $\chi_r$ is the unique primitive quadratic character modulo $r$. Consequently, these newforms satisfy $\lambda_f(p) = \lambda_g(p)$ for every prime $p$ not dividing $q$ for which $\chi_r(p) = 1$.

Note that when $p$ divides $q$, the Hecke eigenvalues $\lambda_f(p)$ of a Maa\ss{} newform $f \in \BB_0^{\ast}\left(\Gamma_0(q)\right)$ are highly prescribed; we have that $\lambda_f(p) \in \left\{p^{-1/2}, -p^{-1/2}\right\}$ if $p \parallel q$, while $\lambda_f(p) = 0$ if $p^2 \mid q$. Consequently, one can show the following.

\begin{proposition}
Let $q$ be a positive odd non-squarefree integer. There exist distinct newforms $f,g \in \BB_0^{\ast}\left(\Gamma_0(q)\right)$ such that $\lambda_f(p) = \lambda_g(p)$ for every $p < n_0(q)$, where
\[n_0(q) = \max_{\substack{r' \mid r \\ r' > 1}} \left\{\min\left\{p : \chi_{r'}(p) = -1,\ p^2 \nmid q\right\}\right\},\]
with $r$ the largest squarefree integer such that $r^2$ divides $q$.
\end{proposition}

Via the work of Graham and Ringrose \cite[Theorem 1]{Graham} on the least quadratic nonresidue modulo a prime $p$, there exist infinitely many primes $p$ such that for $q = p^2$,
\[n_0(q) \gg \log q \log \log \log q.\]

A natural question to ask is whether this is the longest string of initial Hecke eigenvalues that two distinct newforms can share. To rephrase, what is the smallest number $n_0(q)$ such that if two newforms $f,g \in \BB_0^{\ast}\left(\Gamma_0(q)\right)$ have equal Hecke eigenvalues $\lambda_f(p) = \lambda_g(p)$ for every prime $p \leq n_0(q)$, then $f = g$? We conjecture that, excluding Maa\ss{} newforms with Laplacian eigenvalue $1/4$, twisting by quadratic characters provides the only obstruction to distinguishing Maa\ss{} newforms by their Hecke eigenvalues; this is analogous to \cite[Conjecture 2]{Rajan}, in which a conjecture on distinguishing newforms by their Hecke eigenvalues at a positive density of primes is posed.

\begin{conjecture}
Let $q$ be a positive odd non-squarefree integer, and let $n_0(q)$ denote the smallest number such that if two newforms $f,g \in \BB_0^{\ast}\left(\Gamma_0(q)\right)$ with Laplacian eigenvalues greater than $1/4$ have equal Hecke eigenvalues $\lambda_f(p) = \lambda_g(p)$ for every prime $p \leq n_0(q)$, then $f = g$. Then
\[n_0(q) = \max_{\substack{r' \mid r \\ r' > 1}} \left\{\min\left\{p : \chi_{r'}(p) = -1,\ p^2 \nmid q\right\}\right\},\]
where $r$ is the largest squarefree integer such that $r^2$ divides $q$.
\end{conjecture}

These questions are perhaps more naturally posed for holomorphic newforms. To this end, we let $\BB_k^{\ast}\left(\Gamma_0(q)\right)$ denote the set of holomorphic newforms of even weight $k \geq 2$, level $q \geq 1$, and principal character, and let $\BB_k^{\ast}\left(\Gamma_0(q)\right)_{\nonmon(\chi_r)}$ denote the subset of $\BB_k^{\ast}\left(\Gamma_0(q)\right)$ consisting of newforms whose twist by a quadratic character $\chi_r$ is a different holomorphic newform of the same weight, level, and character. The methods developed in this paper, with some modifications, can be used to show the following.

\begin{theorem}
Let $q$ and $r$ be positive integers with $r > 1$ odd and squarefree. Let $\chi_r$ denote the unique primitive quadratic character modulo $r$. Then if $r^2$ does not divide $q$,
\[\BB_k^{\ast}\left(\Gamma_0(q)\right)_{\nonmon(\chi_r)} = \emptyset,\]
whereas if $r^2$ divides $q$, we have that
\[\frac{\# \BB_k^{\ast}\left(\Gamma_0(q)\right)_{\nonmon(\chi_r)}}{\# \BB_k^{\ast}\left(\Gamma_0(q)\right)} = \prod_{\substack{p \mid r \\ p^2 \parallel q}} \left(1 - \frac{p}{p^2 - p - 1}\right) + o_q(1)\]
as $k$ tends to infinity over the even integers, where the error term depends only on $q$.

Moreover, the same holds if we replace $\BB_k^{\ast}\left(\Gamma_0(q)\right)_{\nonmon(\chi_r)}$ by
\[\bigcap_{\substack{r' \mid r \\ r' > 1}} \BB_k^{\ast}\left(\Gamma_0(q)\right)_{\nonmon(\chi_{r'})}.\]
\end{theorem}

Consequently, we can ask the same questions for holomorphic newforms with essentially the same answers.

\begin{proposition}\label{antiMaeda}
Let $q$ be a positive integer, and let $p$ be a fixed prime not dividing $q$. Let $m(q)$ denote the number of odd squarefree divisors $r' > 1$ of $q$ such that ${r'}^2$ divides $q$ and that $\chi_{r'}(p) = 1$, where $\chi_{r'}$ denotes the unique primitive quadratic character modulo $r'$. Then as $k$ tends to infinity over the even integers, a positive proportion of the Hecke eigenvalues $\lambda_f(p)$ of holomorphic newforms $f \in \BB_k^{\ast}\left(\Gamma_0(q)\right)$ have multiplicity at least $m(q) + 1$.
\end{proposition}

This is of particular interest due to its relation to Maeda's conjecture, which states that for every positive even integer $k$ and every prime $p$, the characteristic polynomial of the Hecke operator $T_p$ acting on the space $\Sc_k\left(\SL_2(\Z)\right)$ of holomorphic cusp forms of weight $k$ and level $1$ is irreducible over $\Q$ and the Galois group of its splitting field is the full symmetric group of order $\dim \Sc_k\left(\SL_2(\Z)\right)!$. In particular, Maeda's conjecture implies that the Hecke eigenvalues $\lambda_f(p)$ of holomorphic newforms $f \in \BB_k\left(\SL_2(\Z)\right)$ of weight $k$ and level $1$ are always simple.

For holomorphic newforms of level $q > 1$, it is known that Maeda's conjecture is false; already for $k = 2$, $q = 23$, $p = 13$, the characteristic polynomial of the Hecke operator $T_p$ acting on $\Sc_k^{\ast}\left(\Gamma_0(q)\right)$ is $(x - 3)^2$. \hyperref[antiMaeda]{Proposition \ref*{antiMaeda}} quantifies this by showing that the characteristic polynomial of a Hecke operator is often not separable in the non-squarefree level setting.

We may also consider the problem of distinguishing holomorphic newforms by their Hecke eigenvalues, an issue recently studied by Chow and Ghitza \cite{Chow}. For many small positive squarefree integers $q$ and positive even integers $k$, they numerically calculate the least integer $n_0(q,k)$ such that if $f$ and $g$ are holomorphic newforms of weight $k$, level $q$, and principal character with equal Hecke eigenvalues $\lambda_f(p) = \lambda_g(p)$ for all $p \leq n_0(q,k)$, then $f = g$. Again, if $p \parallel q$, then $\lambda_f(p) \in \left\{p^{k/2 - 1}, -p^{k/2 - 1}\right\}$, while if $p^2 \mid q$, then $\lambda_f(p) = 0$, so we should not expect to be able to distinguish newforms by their Hecke eigenvalues at primes dividing the level (though note that Atkin--Lehner operators determine the sign of $\lambda_f(p)$ when $p \parallel q$). Based on their calculations, Chow and Ghitza suggest that for sufficiently large $k$ independent of the level $q$, the issue of Hecke eigenvalues at primes dividing the level is the only obstruction for squarefree levels.

\begin{conjecture}[Stability Conjecture for Holomorphic Newforms {\cite[Conjecture 4.1]{Chow}}]
Let $q$ be a positive squarefree integer and let $k$ be a positive even integer. Then there exists a positive integer $k_0$ such that for all $k \geq k_0$, $n_0(q,k)$ is equal to the smallest prime not dividing $q$.
\end{conjecture}

When $q$ is not squarefree, their numerical evidence suggests that this conjecture must be altered; in particular, for $q = 49$ it seems that $n_0(q,k) = 3$ for all even $k \geq 4$. This is due to the presence of distinct newforms $f,g$ of level $49$ for which $f \otimes \chi_7 = g$; as $\chi_7(2) = 1$, these newforms satisfy $\lambda_f(2) = \lambda_g(2)$. We propose the following modification to this conjecture for non-squarefree levels.

\begin{conjecture}
Let $q$ be a positive odd non-squarefree integer and let $k$ be a positive even integer. Then there exists a positive integer $k_0$ such that for all $k \geq k_0$,
\[n_0(q,k) = \max_{\substack{r' \mid r \\ r' > 1}} \left\{\min\left\{p : \chi_{r'}(p) = -1,\ p^2 \nmid q\right\}\right\},\]
where $r$ is the largest squarefree integer such that $r^2$ divides $q$.
\end{conjecture}

This conjecture agrees with the data in \cite[Appendix]{Chow} --- for example, that $n_0(225,k) = 7$ for $k \in \{2,4,6,8,10,12\}$ --- with the caveat that this data is limited to small values of $q$ and $k$.

\subsection{Twists of \texorpdfstring{Maa\ss{}}{Maa\80\337} Forms on \texorpdfstring{$\Gamma_1(q) \backslash \Hb$}{\textGamma\9040\201(q)\textbackslash H} and \texorpdfstring{$\Gamma(q) \backslash \Hb$}{\textGamma(q)\textbackslash H}}\label{TwistsGamma1(q)Gamma(q)sect}

In this article, we classified nonmonomial Maa\ss{} newforms of weight $0$ and principal character that are level- and character-invariant under twisting by a Dirichlet character, with such characters necessarily being quadratic. A related question is to relax the character invariance and classify the larger family of weight $\kappa$ Maa\ss{} newforms, possibly having nonprincipal character, that are level-invariant under twisting by some Dirichlet character; here $\kappa \in \{0,1\}$. More precisely, we let
\[\BB_{\kappa}^{\ast}\left(\Gamma_1(q)\right) \defeq \bigsqcup_{\substack{\chi \hspace{-.15cm} \pmod{q} \\ \chi(-1) = (-1)^{\kappa}}} \BB_{\kappa}^{\ast}(q,\chi)\]
denote the set of weight $\kappa$ Maa\ss{} newforms on $\Gamma_1(q) \backslash \Hb$, once again Hecke-normalised. We ask when two distinct newforms $f,g \in \BB_{\kappa}^{\ast}\left(\Gamma_1(q)\right)$ are related by $f \otimes \chi = g$ for some Dirichlet character $\chi$.

Classifying these newforms would lead to a better understanding of the new part of the cuspidal spectrum of the weight $\kappa$ Laplacian
\[\Delta_{\kappa} = -y \left(\frac{\dee^2}{\dee x^2} + \frac{\dee^2}{\dee y^2}\right) + i\kappa y \frac{\dee}{\dee x}\]
on $\Gamma_1(q) \backslash \Hb$, which of course contains the new part of the cuspidal spectrum of the weight $\kappa$ Laplacian on $\Gamma_0(q) \backslash \Hb$. The chief difference in this setting is that spectral multiplicity occurs already for squarefree values of $q$, as was first observed by Booker and Str\"{o}mbergsson \cite[Section 3.4]{Booker}. Indeed, if $q \notin \{1,2,3,5,6,10\}$ is squarefree, then there exists an even Dirichlet character $\chi$ modulo $q$ satisfying $\chi \neq \overline{\chi}$, and we have that
\[\BB_0^{\ast}\left(q,\chi;\lambda\right) \otimes \overline{\chi} = \BB_0^{\ast}\left(q,\overline{\chi};\lambda\right).\]
Similarly, if $q \notin \{1,2,3,6\}$ is squarefree, then there exists an odd Dirichlet character $\chi$ modulo $q$ satisfying $\chi \neq \overline{\chi}$, and we have that
\[\BB_1^{\ast}\left(q,\chi;\lambda\right) \otimes \overline{\chi} = \BB_1^{\ast}\left(q,\overline{\chi};\lambda\right).\]
Another notable difference to $\Gamma_0(q) \backslash \Hb$ is that if $q = p^m$ is a power of an odd prime, then the largest possible multiplicity of an eigenvalue can be shown to grow as the power $m$ grows.

\begin{proposition}
Let $q = p^m$ be a power of an odd prime $p$ with $m \geq 4$. Then for $\kappa \in \{0,1\}$, the new part of the cuspidal spectrum of the weight $\kappa$ Laplacian on $\Gamma_1(q) \backslash \Hb$ contains eigenvalues of multiplicity at least $p^{\lfloor m / 2 \rfloor - 2}(p - 1)^2$.
\end{proposition}

\begin{proof}[Sketch of proof]
Take a newform for which the local component $\pi_p$ at $p$ of the associated cuspidal automorphic representation $\pi$ is a principal series representation
\[\pi_p \cong \beta_{p,1} |\cdot|_p^{s_1} \boxplus \beta_{p,2} |\cdot|_p^{s_2},\]
where $\beta_{p,1}, \beta_{p,2}$ are characters of $\Z_p^{\times}$ of conductor exponents $c(\beta_{p,1}) = \lceil m/2 \rceil$, $c(\beta_{p,2}) = \lfloor m/2 \rfloor$, and we choose $\beta_{p,2} = \beta_{p,1}$ if $m$ is even; the conductor exponent of $\pi_p$ is
\[c(\pi_p) = c(\beta_{p,1}) + c(\beta_{p,2}) = \left\lceil \frac{m}{2} \right\rceil + \left\lfloor \frac{m}{2} \right\rfloor = m.\]
For each character $\beta_p'$ of $\Z_p^{\times}$ of conductor exponent at most $\lfloor m/2 \rfloor$ satisfying $c(\beta_{p,2} {\beta_p'}) = \lfloor m/2 \rfloor$, the twist $\pi_p \otimes \beta_p'$ is a principal series representation
\[\beta_{p,1} \beta_p' |\cdot|_p^{s_1} \boxplus \beta_{p,2} \beta_p' |\cdot|_p^{s_2}\]
of conductor exponent
\[c(\beta_{p,1} \beta_p') + c(\beta_{p,2} \beta_p') = \left\lceil \frac{m}{2} \right\rceil + \left\lfloor \frac{m}{2} \right\rfloor = m,\]
as $c(\beta_p') \leq \lfloor m/2 \rfloor$ implies that $c(\beta_{p,1} \beta_p') = c(\beta_{p,1})$. It is clear that no two such twists by different characters are isomorphic. The number of possible twists is the number of characters $\beta_{p,2} \beta_p'$ of $\Z_p^{\times}$ of conductor exponent $\lfloor m / 2 \rfloor$, which is the number of primitive characters modulo $p^{\lfloor m / 2 \rfloor}$; this is $p^{\lfloor m / 2 \rfloor - 2}(p - 1)^2$ if $m \geq 4$.
\end{proof}

Finally, we may also study the larger family $\Ac_{\kappa}(\Gamma(q))$ of Maa\ss{} cusp forms of weight $\kappa \in \{0,1\}$ that are invariant under the action of the principal congruence subgroup
\[\Gamma(q) \defeq \left\{\begin{pmatrix} a & b \\ c & d \end{pmatrix} \in \SL_2(\Z) : a,d \equiv 1 \hspace{-.2cm} \pmod{q}, \ b,c \equiv 0 \hspace{-.2cm} \pmod{q}\right\}.\]
This space has the decomposition
\[\Ac_{\kappa}(\Gamma(q)) = \bigoplus_{\substack{\chi \hspace{-.25cm} \pmod{q} \\ \chi(-1) = (-1)^{\kappa}}} \iota_{q^{-1}} \Ac_{\kappa}(q^2,\chi),\]
where $\iota_{q^{-1}} f(z) \defeq f(q^{-1}z)$; see \cite[Section 3.5]{Humphries}. Thus a natural definition of the subset of newforms in $\Ac_{\kappa}(\Gamma(q))$ is the set
\[\BB_{\kappa}^{\ast}(\Gamma(q)) \defeq \bigsqcup_{\substack{\chi \hspace{-.25cm} \pmod{q} \\ \chi(-1) = (-1)^{\kappa}}} \iota_{q^{-1}} \BB_{\kappa}^{\ast}(q^2,\chi).\]

Take for simplicity $q$ to be an odd prime $p$, so that for $\chi$ of conductor $p$, the local component $\pi_p$ at $p$ of the cuspidal automorphic representation $\pi$ associated to a newform $f \in \BB_{\kappa}^{\ast}(p^2,\chi)$ is a special representation $\beta_p |\cdot|_p^s \St_p$ with $c(\beta_p) = 1$, so that $c(\pi_p) = 2$. Given a character $\beta_p'$ of $\Z_p^{\times}$ with $c(\beta_p') = 1$, we have that $\pi_p \otimes \beta_p' = \beta_p {\beta_p'}^2 |\cdot|_p^s \St_p$, which has conductor exponent $2$ unless $\beta_p = {\beta_p'}^{-2}$, in which case it has conductor exponent $1$. There are $(p - 1)/2$ characters of $\Z_p^{\times}$ of conductor exponent at most $1$ that are squares of characters $\beta_p'$ of $\Z_p^{\times}$ of conductor exponent $1$. Thus there are at least $(p - 3)/2$ Dirichlet characters $\psi$ modulo $p$ such that $\iota_{p^{-1}} f \otimes \psi \in \BB_{\kappa}^{\ast}(\Gamma(p))$, with no two twists coinciding. Noting that the Laplacian eigenvalue of $\iota_{q^{-1}} f$ is the same as that of $f$, we see that for $\kappa \in \{0,1\}$, the new part of the cuspidal spectrum of the weight $\kappa$ Laplacian on $\Gamma(p) \backslash \Hb$ contains eigenvalues of multiplicity at least $(p - 3)/2$; cf.\ \cite{Randol}.

We leave unaddressed the complete classification of local representations that are conductor-invariant under twists, though the methods we developed in this article would certainly be capable of dealing with this problem. The main new issue would be classifying supercuspidal representations that are conductor-invariant under twisting by a character that is not necessarily quadratic; this would involve significantly more work than the simpler case in \hyperref[supercuspidaltwist]{Lemma \ref*{supercuspidaltwist}}.

\section*{}
\vspace{-.6cm}

\subsection*{Acknowledgements}

The author would like to thank Will Sawin for many helpful discussions with regards to \hyperref[monomialsect]{Section \ref*{monomialsect}}, especially the proof of \hyperref[latticelemma]{Lemma \ref*{latticelemma}}, as well as the referees for many detailed comments and corrections.

\end{document}